\documentclass[a4paper,12pt,leqno]{amsart}
\makeatletter
\usepackage{amsfonts,delarray,amssymb,amsmath,amsthm,a4,a4wide, color}

\usepackage{color}

\title[Convergence of Ohta-Kawasaki to nonlocal Mullins-Sekerka Motion]
{On the convergence of the Ohta-Kawasaki Equation to motion by nonlocal Mullins-Sekerka Law}
\author{Nam Q.  Le}
\address{Department of Mathematics, Columbia University, New York,  NY 10027, USA}
\email{namle@math.columbia.edu}

\newcommand{\review}[2][\right]{\relax
\ifx#1\right\relax \left.\fi#2#1\rvert}
\let\abs=\envert
 
\newtheorem{theorem}{Theorem}[section]
\newtheorem{propo}{Proposition}[section]
\newtheorem{remark}{Remark}[section]

\newtheorem{lemma}{Lemma}[section]
\newcommand{\bef}{\begin{flushright}}
\newcommand{\eef}{\end{flushright}}

\newcommand{\eval}[2][\right]{\relax
\ifx#1\right\relax \left.\fi#2#1\rvert}

\let\abs=\envert

\numberwithin{equation}{section}
\let\norm=\enVert
\newcommand\e{\varepsilon}
\newcommand{\h}{\hspace*{.24in}}

\def\h{\hspace*{.24in}}

\def\beq{\begin{eqnarray*}}
\def\eeq{\end{eqnarray*}}

\newenvironment{myindentpar}[1]%
{\begin{list}{}%
         {\setlength{\leftmargin}{#1}}%
         \item[]%
}
{\end{list}}

\def\RR{\mbox{$I\hspace{-.06in}R$}}

\begin{document}
\date{April 24, 2010}

\maketitle
\author

\pagenumbering{arabic}
\begin{abstract}
In this paper, we establish the convergence of the Ohta-Kawasaki equation to motion by nonlocal Mullins-Sekerka law 
on any smooth domain in 
space dimensions $N\leq 3$. These equations arise in modeling microphase separation in diblock copolymers. The only 
assumptions that guarantee our convergence result are $(i)$ well-preparedness of the initial data and $(ii)$ smoothness of 
the limiting interface. Our method makes use of the 
``Gamma-convergence'' of gradient flows scheme initiated by Sandier and Serfaty and the constancy of multiplicity of the 
limiting interface due 
to its smoothness. For the case of radially symmetric initial data without well-preparedness, we give a new and short proof 
of the result of M. Henry for all space dimensions. Finally, we establish transport estimates for solutions of the Ohta-Kawasaki equation 
characterizing their transport mechanism.
\end{abstract}
%\tableofcontents
\section{Introduction}
\subsection{The Ohta-Kawasaki equation} This paper is concerned with the asymptotic limit, as $\varepsilon\searrow 0,$ of 
the solutions to the Ohta-Kawasaki equation \cite{OK} with initial data $u^{\e}_{0}$
 \begin{equation}\displaystyle
  \left\{
  \begin{alignedat}{2}
   \partial_{t}u^{\varepsilon} &= -\Delta w^{\varepsilon}   \h\h&& (x, t)\in \Omega\times (0, \infty) 
\\\ w^{\varepsilon}&=\varepsilon \Delta u^{\varepsilon}-\varepsilon^{-1}f(u^{\varepsilon}) -\lambda v^{\e} \h\h&& (x, t)\in \Omega\times [0, \infty)
\\\ 
-\Delta v^{\e}& = u^{\e} -\overline{u^{\e}}_{\Omega} \h\h&& (x, t)\in \Omega\times [0, \infty) 
\\\ 
\overline{v^{\e}}_{\Omega} &=0 \h&& (x, t)\in \Omega\times [0, \infty) \\\
\frac{\partial u^{\varepsilon}}{\partial n} ~(x, t)& = \frac{\partial v^{\varepsilon}}{\partial n} (x, t) 
= \frac{\partial w^{\varepsilon}}{\partial n} ~(x, t)=0 && (x, t)\in \partial\Omega\times [0, \infty)\\\ u^{\varepsilon} (x, 0) & = 
u_{0}^{\varepsilon} ~(x) && x\in\Omega . 
  \end{alignedat}
  \right.
\label{OKeq}\end{equation}
\h Here $\Omega$ is a bounded smooth domain in $\RR^{N}~(N\geq 2)$, $f(u)=2u(u^2-1)$ is the derivative of 
the double-well potential $W(u)=\frac{1}{2}(u^2-1)^2$ and $\lambda\geq 0$ is a fixed constant. Throughout, we denote $\overline{u}_{\Omega}$ the average of a function $u$ over $\Omega$:
%\begin{equation*}
 $\displaystyle\overline{u}_{\Omega} =\frac{1}{\abs{\Omega}} \int_{\Omega} u dx.$
%\end{equation*}
Moreover, for any function $u$ with average zero, we denote by $\norm{u}_{H^{-1}(\Omega)} = \norm{\nabla \Delta^{-1} u}_{L^{2}},$ where $\Delta^{-1} u$ is the unique solution of the elliptic problem
\begin{equation*}
\left\{
 \begin{alignedat}{2}
 -\Delta v &= u  \h~&&\text{in}~ \Omega, \\\
 \overline{v_{\Omega}} &= 0 \h~&&\text{in}~ \Omega, \\\
\frac{\partial v}{\partial n} & =0 \h~&&\text{on}~ \partial\Omega.
 \end{alignedat} 
 \right.
 \end{equation*}
 \h Associated with equation (\ref{OKeq}) is the Ohta-Kawasaki energy functional $E_{\e}$ first introduced
in \cite{OK} to model microphase separation in diblock copolymers' melts ({\it cf.} \cite{BF}):
\begin{equation}
 E_{\e} (u) = \int_{\Omega}\left(\frac{\varepsilon}{2}\abs{\nabla u}^2 +\frac{1}{\varepsilon}W(u)\right) dx + 
\frac{\lambda}{2}\norm{u- \overline{u}_{\Omega}}_{H^{-1}(\Omega)}^2.
\label{OKfn}
\end{equation}
 See also \cite{CR} for a  derivation of $E_{\e}$ 
from the statistical physics of interacting block copolymers. A diblock copolymer molecule
is a linear chain consisting of two subchains made of two different monomers, say $A$ and $B$. The function $u^{\e}$ in (\ref{OKeq})
is related to the density parameter describing the diblock copolymers' melts: it is essentially the difference between the averaged
densities of monomers $A$ and $B$. The parameter $\e$ is proportional to the thickness of the transition regions between two 
monomers and $\lambda$ is a parameter related to the polymerization index. Outside the transition regions, $u^{\e}\approx \pm 1$. \\
\h There has been a vast literature on the analysis of (\ref{OKfn}). We refer the reader to \cite{ACO, CO05, CO07, RWIFB} for the study of 
minimizers of (\ref{OKfn}) and \cite{NS, RW02, RW03} for the existence and stability of stationary solutions of (\ref{OKfn}).
\subsection{The nonlocal Mullins-Sekerka law } It is expected \cite{NO} that the Ohta-Kawasaki equation converges to motion by nonlocal Mullins-Sekerka 
law. This means that, as $\varepsilon\searrow 0,$ $(u^{\e}, v^{\varepsilon}, w^{\e})$ tends to a limit $(u^{0}, v, w), $ which, together with a free 
boundary $\displaystyle \cup_{0\leq t\leq T}(\Gamma(t)\times \{t\})$, 
solves the following free-boundary problem in a time interval $[0, T]$ for some $T>0$:
\begin{equation}
\left\{
 \begin{alignedat}{2}
u^{0}& = \pm 1\h~&& \text{in}~\Omega_{t}^{\pm}, ~t\in [0, T], \\\
v &= \Delta^{-1}(u^{0} -\overline{u^{0}_{\Omega}}) \h~&&\text{in}~ \Omega\times [0, T], \\\
 \Delta w & = 0 \h~&&\text{in}~ \Omega\backslash \Gamma(t), ~t\in [0, T], \\\ 
\frac{\partial w}{\partial n} & =0 \h~&&\text{on}~ \partial\Omega\times [0, T], \\\
w & =\sigma \kappa -\lambda v \h~&&\text{on}~\Gamma(t),~t\in[0, T], \\\ 
 \partial_{t}\Gamma & = \frac{1}{2} \left [\frac{\partial w}{\partial n}\right]_ {\Gamma(t)} \h~&&\text{on}~\Gamma(t), ~t\in [0, T], \\\ 
\Gamma(0) &=\Gamma_{0}.
 \end{alignedat} 
 \right.
 \label{NO}
 \end{equation}
 \h Here $\kappa(t)$ is the mean curvature of the closed, connected hypersurface $\Gamma(t)\subset\Omega$ with 
the sign convention that the boundary of a convex domain has positive mean 
curvature (More generally, we will consider in this paper the case $\Gamma(t)$ is the union of a finite number of closed, connected hypersurfaces.); 
$\sigma =\displaystyle \int_{-1}^{1}\sqrt{W(s)/2} ds=\frac{2}{3}$; $\partial_{t}\Gamma$ is 
the normal velocity of the hypersurface $\Gamma(t)$ with the sign convention that the normal 
velocity on the boundary of an expanding domain is positive; $\stackrel{\rightarrow}{n}$ is the unit 
outernormal either to $\Omega$ or $\Gamma(t)$; $\left[\frac{\partial w}{\partial n}\right]_{\Gamma(t)}$ denotes 
the jump in the normal derivative of $w$ through the 
hypersurface $\Gamma(t)$, i.e., $\left[\frac{\partial w}{\partial n}\right]_{\Gamma(t)}=\frac{\partial w^{+}}{\partial n}-\frac{\partial w^{-}}{\partial n}$,
 where $w^{+}$ and $w^{-}$ are respectively the 
restriction of $w$ on $\Omega_{t}^{+}$ and $\Omega_{t}^{-}$, the exterior and 
interior of $\Gamma(t)$ in $\Omega$; and finally, $\Gamma_{0}\subset\subset\Omega$ is the initial 
hypersurface separating the phases of the function $u_{0}\in \text{BV}(\Omega,\{-1, 1\})$ which is the $L^{2}(\Omega)$ limit 
of the sequence $\{u_{0}^{\varepsilon}\}_{0<\varepsilon<1}$ (after extraction). \\
\h Associated with (\ref{NO}) is the nonlocal area functional $E$ defined by
\begin{equation}
 E(u) = \sigma\int_{\Omega} \abs{\nabla u} + 
\frac{\lambda}{2}\norm{u- \overline{u}_{\Omega}}_{H^{-1}(\Omega)}^2 \equiv E(\Gamma)
\label{NOener}
\end{equation}
where $\Gamma$ is the interface separating the phases of the function $u\in BV(\Omega,\{-1,1\})$. This functional consists of competing short-range ($\sigma\int_{\Omega} \abs{\nabla u}$) and long-range ($\frac{\lambda}{2}\norm{u- \overline{u}_{\Omega}}_{H^{-1}(\Omega)}^2$) contributions. The former term is attractive, preferring large domains where $u=\pm 1$ with boundaries of minimal surface area. The latter term is repulsive, favoring small domains where $u=\pm 1$ which lead to cancellations.  \\
\h Let us comment briefly on the well-posedness of (\ref{OKeq}) and (\ref{NO}). For each $\e>0$, one can adapt the method in \cite{ESCH}
to prove the existence and uniqueness of smooth solution to (\ref{OKeq}) for smooth initial data $u^{\e}_{0}$. The existence and uniqueness of classical solution for the free-boundary problem
(\ref{NO}) with smooth
initial data have been 
established in \cite{EN}. 
\subsection{Related and previous results } When $\lambda =0$, (\ref{OKeq}) and (\ref{NO}) are the Cahn-Hilliard equation \cite{Cahn, ESCH, Hilliard} and
Mullins-Sekerka law \cite{MS}, respectively.  The convergence of the Cahn-Hilliard equation to motion
by Mullins-Sekerka law has been established in certain cases: for a class of very well-prepared initial data in \cite{ABC, CCO},
in the presence of spherical symmetry in \cite{Stoth}, for general initial data but for a weak varifold formulation
of the Mullins-Sekerka law in \cite{ChenCH}, and under the validity of an $H^{1}$-version of De Giorgi's conjecture
in \cite{LeCH}. For the sake of completeness, we state here the key ingredient of our $H^{1}$-version of De Giorgi's conjecture 
in \cite{LeCH}:\\
{\bf Conjecture (CH).}  {\it Let $\left\{u^{\varepsilon}\right\}_{0<\varepsilon\leq 1}$ be a sequence of $C^{3}$ functions 
satisfying 
\begin{equation*}
 \int_{\Omega}\left(\frac{\varepsilon}{2}\abs{\nabla u^{\e}}^2 +\frac{1}{\varepsilon}W(u^{\e})\right)dx \leq M<\infty,~~\overline{u^{\e}_{\Omega}}
= m_{\varepsilon}\in(-m,m)~(0<m<1).\end{equation*}
 and let $u\in BV(\Omega,\{-1,1\})$ be its $L^{2}(\Omega)$-limit (after extraction). Assume that 
$\Gamma =\partial^{\ast}\{u=1\}\cap\Omega$ is $C^{2}$ and connected. Then 
$$\displaystyle \liminf_{\e\rightarrow 0}\int_{\Omega}\abs{\nabla (\varepsilon\Delta u^{\varepsilon}-
\varepsilon^{-1}f(u^{\varepsilon}))}^2 dx\geq \sigma^2 \norm{\kappa}^{2}_{H_{n}^{1/2}(\Gamma)}.$$}
In the above conjecture, 
$\partial^{\ast} E$ denotes the reduced boundary of a set $E$ of finite
perimeter and for any 
function $g$ defined on $\Gamma$, we denote by $\norm{g}^{2}_{H_{n}^{1/2}(\Gamma)}$ the square of the homogeneous Sobolev norm of $g$ (see
also Section \ref{grNO})
\begin{equation*}
\norm{g}^{2}_{H_{n}^{1/2}(\Gamma)} = \inf_{w\in H^{1}({\Omega}),~ w= g~on~\Gamma}\int_{\Omega} \abs{\nabla w}^2 dx.
\end{equation*}
\h When $\lambda>0$, there have been very few results justifying the convergence of (\ref{OKeq}) to (\ref{NO}) except in some special cases: in one space dimension by Fife and Hilhorst \cite{FH} and 
in higher dimensions with spherical symmetry by Henry \cite{Henry}. See related results in \cite{HHN}. 
On the other hand, there have been recent interesting works \cite{CP, NieO} on the next order asymptotic limit of small volume fraction of
(\ref{NO}) and (\ref{NOener}). Concerning
dynamics, assuming the initial component of small volume fraction, say $\{u^{0}(0) =1\}$, consists of an ensemble of small spheres, the 
work \cite{NieO} rigorously
derives mean-field models for the evolution of such spheres under the nonlocal Mullins-Sekerka law (\ref{NO}).  \\
\h Note that the proof of convergence of
(\ref{OKeq}) to (\ref{NO}) with spherical symmetry in \cite {Henry} was a nontrivial extension of the proof in \cite{Stoth}
for the Cahn-Hilliard equation. In fact, (\ref{NO}) and Mullins-Sekerka dynamics are quite different. As observed in \cite{EN}, in contrast 
to the Mullins-Sekerka law, (\ref{NO}) does not
necessarily decrease the area of $\Gamma(t)$ and most importantly, spheres are not in general equilibria to (\ref{NO}) except for very special
domains $\Omega$ like spherical ones. It has been an interesting
and challenging problem to rigorously establish the convergence of (\ref{OKeq}) to (\ref{NO}) for general domains in higher space
dimensions. \\
\h  We are motivated
by the question: is there any way to establish the convergence of (\ref{OKeq}) to (\ref{NO}), similar to the convergence
of Cahn-Hilliard to motion by Mullins-Sekerka law, where the smooth nonlocal perturbations 
$v^{\e}$ and $v$ present no essential difficulty? We are also motivated by an open question in Glasner and Choksi \cite{GC} about the
justification of the dynamic equations (\ref{NO}) (which have the gradient flow structure) from (\ref{OKeq}) via the recently
established connection between Gamma-convergence and gradient flows \cite{SS}.\\
\h It turns out that one can, at least formally, follow the ``Gamma-convergence'' of gradient flows scheme initiated by 
Sandier and Serfaty \cite{SS} to prove the convergence of (\ref{OKeq}) to (\ref{NO}) because of the following observations:
\begin{myindentpar}{0.5cm}
{\bf 1.} Equation (\ref{OKeq}) is the $H^{-1}$ gradient flow of the Ohta-Kawasaki functional (see Sect. \ref{grOK}) $E_{\e}.$\\
{\bf 2.} The functional $E_{\e}$ Gamma-converges to the nonlocal area functional $E.$\\
{\bf 3.} Equation (\ref{NO}) is the $H^{-1}$-gradient flow of $E$ (see Sect. \ref{grNO}).
\end{myindentpar}
\h Concerning Gamma-convergence, what we will actually need is only the following liminf inequality in the definition of Gamma-convergence (denoted
by $\Gamma$-convergence in what follows) \cite{Braides}:
\begin{myindentpar}{0.8cm}
\it
For any sequence $u^{\e}$ such that $\displaystyle\limsup_{\e\rightarrow 0} E_{\e}(u^{\e})<\infty$, we can extract a subsequence, still labeled $u^{\e}$,
such that $u^{\e}$ converges in $L^{2}(\Omega)$ to a function $u^{0}\in BV(\Omega, \{-1,1\})$ and
\begin{equation*}
 \liminf_{\e\rightarrow 0} E_{\e}(u^{\e})\geq E(u^{0}).
\end{equation*}
\end{myindentpar}
This inequality is well-known. It is a simple consequence of the $\Gamma$-convergence of the Allen-Cahn functional $\displaystyle 
\int_{\Omega}\left(\frac{\varepsilon}{2}\abs{\nabla u}^2 +\frac{1}{\varepsilon}W(u)\right) dx$ to the area functional $\displaystyle\sigma\int_{\Omega}\abs{\nabla u}$, due to Modica-Mortola \cite{MM} 
(see also \cite{Stern}), combined with the fact that the nonlocal term $
\frac{\lambda}{2}\norm{u- \overline{u}_{\Omega}}_{H^{-1}(\Omega)}^2$ is its continuous perturbation.
\subsection{Main results} In this paper, following the ``Gamma-convergence'' of gradient flows scheme in \cite{SS}, we 
prove the convergence of (\ref{OKeq}) to (\ref{NO}) on any smooth domain in space dimensions $N\leq 3$
under the following assumptions:
\begin{myindentpar}{0.6cm} $(i)$ the initial data is well-prepared\\
and \\
$(ii)$ the limiting interface is smooth.
\end{myindentpar}
 Note that the scheme in \cite{SS} when applied to Ginzburg-Landau equation with a finite number of vortices requires no smoothness of the limiting structure. This is due to its finite dimensionality character. Our setting is infinite dimensional and thus extra regularity is required to make sense of the gradient flow. It would be interesting to establish the smoothness of the limiting interface, maybe under 
some additional assumptions on the general initial data.\\
\h Throughout the paper, we always assume that the initial data $u^{\e}_{0}$ satisfies the mass constraint
\begin{equation}
 \overline{u^{\e}_{0\Omega}} = m_{\e} \in (-m,m) ~(0<m<1).
\label{midmass}
\end{equation}

\h Our first main theorem reads
\begin{theorem} Assume that the space dimensions $N\leq 3$. Let $(u^{\varepsilon}, v^{\varepsilon}, w^{\e})$ be the smooth solution 
of (\ref{OKeq}) on $\Omega\times [0, \infty)$ with 
initial data $u_{0}^{\varepsilon}$. Assume that, after 
extraction, $u_{0}^{\varepsilon}$ converges strongly 
in $L^{2}(\Omega)$ to  $u^{0}(\cdot,0)\in$ BV $(\Omega,\{-1,1\})$ with 
interface $\Gamma(0) = \partial\{x\in\Omega: u^{0}(x,0)=1\}\cap\Omega$ consisting of a finite 
number of closed, connected $C^{3}$ hypersurfaces. 
Let $T_{\ast}>0$ be the minimum of the collision time and of the exit time from $\Omega$ of the hypersurfaces under the nonlocal Mullins-Sekerka law 
(\ref{NO}) with the initial interface $\Gamma(0)$.\\
\h Then, after extraction, we have that for 
all $t\in [0, T_{\ast})$, $u^{\varepsilon}(\cdot,t)$ converges strongly in $L^{2}(\Omega)$ to  $u^{0}(\cdot,t)\in$ BV $(\Omega,\{-1,1\})$ 
with interface $\Gamma(t) = \partial\{x\in\Omega: u^{0}(x,t)=1\}\cap\Omega.$ Moreover, under the following assumptions
\begin{myindentpar}{1cm} (A1) The initial data $u_{0}^{\varepsilon}$ is 
well-prepared, i.e., $\lim_{\varepsilon\rightarrow 0} E_{\varepsilon}(u_{0}^{\varepsilon})= E(u^{0}),$ \\
 (A2) $\cup_{t\in[0,T_{\ast})}(\Gamma(t)\times t)$ is a $C^{3,\alpha}$ ($\alpha>0$) space-time hypersurface, that is, this 
hypersurface is $C^{\alpha}$ in time
and for each $t\in [0, T_{\ast})$, $\Gamma(t)$ is $C^{3}$,
\end{myindentpar}
the Ohta-Kawasaki equation converges to motion by nonlocal Mullins-Sekerka law. That is, $w^{\varepsilon}$ converges 
strongly in $L^{2}((0, T_{\ast}), H^{1}(\Omega))$ to $w$ solving (\ref{NO}) with the initial interface $\Gamma(0)$. 

\label{dynamics}
\end{theorem}
%\section{Methodology and Remarks on the Proof}
\begin{remark}
The restriction $N\leq 3$ on the space dimension enables us to apply Tonegawa's convergence theorem \cite{TonePhase} for diffused interface
whose chemical potential belongs to $W^{1,p}(\Omega)$ with $p> \frac{N}{2}.$ See the proof of Proposition \ref{mild}. In our case, $p=2$. 
\end{remark}
\begin{remark}
 There is a large class of initial data $u^{\e}_{0}$ for which the solutions to (\ref{OKeq}) satisfy (A1) and (A2). This class includes very well-prepared initial data for 
general domains $\Omega$ constructed 
similarly as in \cite{ABC, CCO} in the context of Cahn-Hilliard equation and radially symmetric initial data for spherical domains $\Omega$. In the later case, the H\"{o}lder continuity in time of $u^{\e}$ (as in (\ref{volholder})) implies the H\"{o}lder continuity in time of $\Gamma(t)$.
\end{remark}
\begin{remark}
 The interface $\Gamma(t)$ is contained in the limit measure $\mu(t)$ of $\left( \abs{\nabla u^{\e}(t)}^2 +
\frac{1}{\varepsilon}W(u^{\e}(t))\right) dx$. Throughout, we use the notation $u^{\e}(t) = u^{\e}(\cdot,t)$ etc.
In general, $supp\mu(t)\backslash \Gamma(t)$ is not empty. The presence of hidden 
boundary outside the interface is responsible for this. However, under $(A1)-(A2)$, hidden boundaries will be prevented 
during the evolution of (\ref{OKeq}). 
\end{remark}
In the process of proving Theorem \ref{dynamics}, we also prove Conjecture {\bf (CH)} for space dimensions $N\leq 3.$ We state here as
\begin{theorem}
 Let $\left\{u^{\varepsilon}\right\}_{0<\varepsilon\leq 1}$ be a sequence of $C^{3}$ functions 
satisfying 
\begin{equation}
 \int_{\Omega}\left(\frac{\varepsilon}{2}\abs{\nabla u^{\e}}^2 +\frac{1}{\varepsilon}W(u^{\e})\right) dx \leq M<\infty,~~\overline{u^{\e}_{\Omega}}
= m_{\varepsilon}\in(-m,m)~(0<m<1).\label{fixedmass}\end{equation}
 and let $u\in BV(\Omega,\{-1,1\})$ be its $L^{2}(\Omega)$-limit (after extraction). Assume that 
$\Gamma =\partial^{\ast}\{u=1\}\cap\Omega$ is $C^{2}$ and connected. Furthermore, assume that the space dimension $N =2$ or $3$. Then the 
following inequality holds
 \begin{equation}\displaystyle \liminf_{\e\rightarrow 0}\int_{\Omega}\abs{\nabla (\varepsilon\Delta u^{\varepsilon}-
\varepsilon^{-1}f(u^{\varepsilon}))}^2 dx \geq \sigma^2 \norm{\kappa}^{2}_{H_{n}^{1/2}(\Gamma)}.
\label{DGineq}
\end{equation}
\label{conjCH}
\end{theorem}

 For the case of radially symmetric initial data without well-preparedness, we give a new and short proof of the result of Henry \cite{Henry} 
for all space dimensions  in our next main 
theorem
\begin{theorem}
 Assume that the space dimensions $N\geq 2$ and $\Omega = B_{1}\subset \RR^{N}$. Let $(u^{\varepsilon}, v^{\varepsilon}, w^{\e})$ be the smooth solution 
of (\ref{OKeq}) on $\Omega\times [0, \infty)$ with 
radially symmetric initial data $u_{0}^{\varepsilon}$. Assume that, after extraction, $u_{0}^{\varepsilon}$ converges strongly 
in $L^{2}(\Omega)$ to  $u^{0}(\cdot,0)\in$ BV $(\Omega,\{-1,1\})$ with 
interface $\Gamma(0) = \partial\{x\in\Omega: u^{0}(x,0)=1\}\cap\Omega$ consisting of a finite 
number of spheres. We assume that
\begin{myindentpar}{1cm}
(B) the initial data $u^{\e}_{0}$ has uniformly bounded energy $E_{\e}(u^{\e}_{0})\leq M <\infty$,\\
(BC) there exist $\alpha,\delta, \e_{0}>0$ such that for $\e\leq \e_{0}$, $\abs{u^{\e}_{0}(x)}\geq \alpha$ for 
$x\in S_{\delta}:= \{x\in \Omega: dist(x,\partial\Omega)\leq \delta\}$.
 \end{myindentpar}
Then there exists $T_{\ast}>0$ such that, after extraction, we have that for 
all $t\in [0, T_{\ast})$, $u^{\varepsilon}(\cdot,t)$ converges strongly in $L^{2}(\Omega)$ to  $u^{0}(\cdot,t)\in$ BV $(\Omega,\{-1,1\})$ with 
interface $\Gamma(t) = \partial\{x\in\Omega: u^{0}(x,t)=1\}\cap\Omega$ 
and (\ref{OKeq}) converges to (\ref{NO}) on the time interval $[0, T^{\ast})$. In fact, $T_{\ast}$ can be chosen to be the minimum of the collision time 
and of the exit time from $\Omega$ of the spheres under the nonlocal Mullins-Sekerka law (\ref{NO}) with initial interface $\Gamma(0)$.
\label{sdynamics}
\end{theorem}
\begin{remark}
 We are not seeking optimal conditions on the initial data $u^{\e}_{0}$ to make the proof more transparent. In fact, (BC) can be replaced by
the following condition
\begin{myindentpar}{1cm}
(BC') The limit measure $\mu(0)$ of $\left(\frac{\varepsilon}{2}\abs{\nabla u^{\e}_{0}}^2 +\frac{1}{\varepsilon}W(u^{\e}_{0})\right)dx$
(in the sense of Radon measures) does not concentrate on the boundary $\partial\Omega$: $\mu(0)(\partial\Omega) =0$.
\end{myindentpar}
\end{remark}

As a by-product of our proofs and inspired by a deformation argument in \cite{SS}, we are able to provide a transport
estimate for the Ohta-Kawasaki equation by establishing a convergence of the velocity in its natural energy space. For this purpose, we need a new function space
$H^{-1}_{n}(\Omega)$. It is a modification of the usual $H^{-1}(\Omega)$ and defined as follows.\\
\h Let $\langle,\rangle$ 
denote the pairing between $(H^{1}(\Omega))^{\ast}$ and $H^{1}(\Omega)$. Then, define
\begin{multline*}
H_{n}^{-1}(\Omega) = \{f\in (H^{1}(\Omega))^{\ast}\mid \exists~ g\in H^{1}(\Omega)~\text{such that}~ \langle f,\varphi\rangle = 
\int_{\Omega}\nabla g\cdot\nabla\varphi dx ~~~\forall~ \varphi\in  H^{1}(\Omega)\}.
\end{multline*} 
\h The function $g$ in the above definition is unique up to a constant. We denote by $-\Delta_{n}^{-1} f$ the one with mean $0$ over $\Omega$.
Then, $H_{n}^{-1}(\Omega)$ is a Hilbert space with inner product 
\begin{equation*}
<u,v>_{H_{n}^{-1}(\Omega)} = \int_{\Omega} \nabla(\Delta_{n}^{-1} u)\cdot\nabla(\Delta_{n}^{-1} v) dx~~~\forall~ u,v\in H_{n}^{-1}(\Omega).
\end{equation*}
\h Our final main result states
\begin{theorem}
Let $(u^{\varepsilon}, v^{\varepsilon}, w^{\e})$ be the smooth solution 
of (\ref{OKeq}) on $\Omega\times [0, \infty)$ as in Theorem \ref{dynamics} or Theorem
\ref{sdynamics}. Let $u^{0}(\cdot,t)$ be the limit in $L^{2}(\Omega)$ of $u^{\e}(\cdot,t)$ with smooth interface $\Gamma(t)$ 
satisfying (\ref{NO}). Let $\partial_{t}{\bf\Gamma}\in (C^{1}_{c}(\Omega))^{N}$ be any smooth extension of $(\partial_{t}\Gamma)\stackrel{\rightarrow}{n} $ where 
$\stackrel{\rightarrow}{n}$ is the unit outernormal to $\Gamma(t)$. Then we can find a small perturbation $\partial_{t}{\bf \Gamma}^{\e}$ of 
$\partial_{t}{\bf \Gamma}$ such that
\begin{equation}
 \lim_{\e\rightarrow 0}\norm{\partial_{t}{\bf \Gamma}^{\e}-\partial_{t}{\bf \Gamma} }_{C^{1}_{0}(\Omega)} =0~\text{for
each time slice } t \geq 0
\label{smallde}
\end{equation}
and
\begin{equation}
 \lim_{\e\rightarrow 0}\int_{t_{1}}^{T^{\ast}}\norm{\partial_{t}u^{\e} + \partial_{t}{\bf \Gamma}^{\e}\cdot \nabla u^{\e}}^2_{H^{-1}_{n}(\Omega)}
 dt =0 ~\text{for
all } t_{1} >0.
\label{veloconv}
\end{equation}
In the case of well-prepared initial data, (\ref{veloconv}) also holds for $t_{1}=0$.
\label{transport}
\end{theorem}
\begin{remark}
To our knowledge, in the context of the Cahn-Hilliard and Ohta-Kawasaki equations, the transport estimate (\ref{veloconv}) is new. It expresses that $u^{\e}$ is very close to being simply transported at the velocity
$\partial_{t}\Gamma$ around $\Gamma$. The space $L^{2}((0, T^{\ast}), H^{-1}_{n}(\Omega))$ is the natural energy space
for the velocity $\partial_{t}u^{\e}.$ From the definition of $H^{-1}_{n}(\Omega)$(see also section \ref{grOK}), we have
\begin{equation*}
 \int_{0}^{T^{\ast}}\norm{\partial_{t}u^{\e} }^2_{H^{-1}_{n}(\Omega)} dt =\int_{0}^{T^{\ast}} \norm{\nabla w^{\e}(t)}^{2}_{L^{2}(\Omega)} dt
= E_{\e}(u^{\e}(0)) - E_{\e}(u^{\e}(T^{\ast})) \leq M.
\end{equation*}
\end{remark}
\begin{remark}
 In general, $\partial_{t}{\bf \Gamma}\cdot\nabla u^{\e}$ does not belong to $H^{-1}_{n}(\Omega)$. Thus, we need a small perturbation 
$\partial_{t}{\bf \Gamma}^{\e}$ of
$\partial_{t}{\bf \Gamma}$ as in (\ref{smallde}) such that $\partial_{t}{\bf \Gamma}^{\e}\cdot \nabla u^{\e}\in H^{-1}_{n}(\Omega).$ 
\end{remark}
\begin{remark}
Setting $\lambda =0$ in Theorems \ref{dynamics}, \ref{sdynamics}\&\ref{transport}, we recover convergence 
results for the Cahn-Hilliard equation to motion by Mullins-Sekerka law. Note that, due to the validity of Conjecture ({\bf CH}) established
in Theorem \ref{conjCH} for space dimensions $N\leq 3$, we are able to remove condition $(A3)$ of Theorem 1.3 in our previous paper
\cite{LeCH}.
\end{remark}
\subsection{Ideas of the proofs}
We conclude this introduction with some remarks on the proofs of the main theorems. 
\begin{myindentpar}{0.2cm}
{\bf 1.} \begin{myindentpar}{0.2cm} (i)The structure of the proof
of Theorem \ref{dynamics} is essentially the same as 
that of the convergence of Cahn-Hilliard equation 
to motion by Mullins-Sekerka law in \cite{LeCH} with the nonlocal term added. However, the main ingredient and difficulty, Lemma \ref{mainlemma}, is not assumed as 
it was in Theorem 1.3 of \cite{LeCH}. To prove this lemma, we make use of Tonegawa's convergence theorem, Theorem \ref{Toneconvthm}; R\"{o}ger's locality theorem, Theorem
\ref{Rogerlocality}; and finally, Sch\"{a}tzle's constancy
theorem, Theorem \ref{Schatzle}, on the multiplicity of the smooth limiting interface. Our proof reveals that the fundamental 
difference between (\ref{OKeq}) and the Cahn-Hilliard equation lies in the potential higher multiplicity of the 
short-range contribution in $E_{\e}$. Precisely speaking, in the limit as $\e\rightarrow 0$ (after extraction), $u^{\e}(t)\rightarrow u^{0}(t) \in BV(\Omega, \{-1,1\})$, the long-range contribution always has multiplicity one, i.e, 
\begin{equation*}
\lim_{\e\rightarrow 0}\frac{\lambda}{2}\norm{u^{\e}(t)- \overline{u^{\e}(t)}_{\Omega}}_{H^{-1}(\Omega)}^2 = \frac{\lambda}{2}\norm{u^{0}(t)- \overline{u^{0}(t)}_{\Omega}}_{H^{-1}(\Omega)}^2.
\end{equation*}
Meanwhile, the short-range contribution may have higher multiplicity, that is,
\begin{equation*}
\lim_{\e\rightarrow 0}\int_{\Omega}\left(\frac{\varepsilon}{2}\abs{\nabla u^{\e}(t)}^2 +\frac{1}{\varepsilon}W(u^{\e}(t))\right) dx = m(t)\sigma \int_{\Omega}\abs{\nabla u^{0}(t)}.
\end{equation*}
Here the multiplicity $m(t)$ is an odd integer, possibly larger than $1$. The statement of Lemma 4.1 is only
true for $m(t) =1$. See also Remark \ref{DGrk}. If $m(t)>1$, which corresponds to the case $u^{\e}(t)$ folds $m(t)$ times 
around the interface $\Gamma(t)$, then our 
approach using the scheme in \cite{SS} completely breaks down. \\
(ii) As mentioned above, the proof of Lemma \ref{mainlemma} only works for single
multiplicity ($m(t)=1$) of the limiting interface and for short time. Similar
result in the Cahn-Hilliard case (see Theorem 1.2 in \cite{LeCH} or Theorem \ref{conjCH} in this paper) works for any constant multiplicity and long
time. Nevertheless, we are able to get around this higher multiplicity issue. Our idea is to use the time continuity
of the limiting interface to 
prove single multiplicity of the short-range contribution for short time, thus establishing Lemma \ref{mainlemma}. Then, to prove Theorem \ref{dynamics},
we will first use 
the $\Gamma$-convergence scheme to prove well-preparedness of solution to (\ref{OKeq}) for short time. The process will be iterated until 
the hypersurfaces in the interface $\Gamma(t)$ collide or exit to the boundary. \\
(iii) Our proof of inequality (\ref{dissipation}) in Lemma \ref{mainlemma} relies heavily on the well-preparedness of the initial data. In the
original gradient flows scheme \cite{SS} and for the local evolution laws like Allen-Cahn and Cahn-Hilliard, we do not have to resort to dynamics (see (C2) in Section
\ref{generalfr} and Theorem \ref{conjCH}). With the presence of the nonlocal terms, a purely static statement similar to 
(\ref{dissipation}) may be false except when the multiplicity one theorem of R\"{o}ger-Tonegawa \cite{RT} can be improved to the case of 
$W^{1,p}$ ($N/2 <p\leq N$) chemical potentials. As far as we know, this issue has not been resolved yet. 
\end{myindentpar}
{\bf 2.} In Theorem \ref{sdynamics}, 
%we do not require any well-preparedness of the initial 
%data nor any restrictions on space dimensions,
the crucial observation that allows us to apply the $\Gamma$-convergence of gradient flows scheme
is that, in the presence of spherical symmetry, the evolution equation (\ref{OKeq}) creates well-preparedness of the evolving interface almost
instantaneously. See (\ref{Fatou}) and Theorem \ref{instancethm}.\\
{\bf 3.} The proof of Theorem \ref{transport} is based on the well-preparedness in time of the evolving interface and a 
deformation argument presented in Proposition \ref{deform}. Its
basic idea is to ``lift'' a curve in the limiting space to a curve in the original space in such a way that the slope
of the lifted curve is that of the original one, and that the energy decreases by that of the 
limiting energy; see (\ref{veloineq}) and (\ref{slopeineq}). This deformation argument was first proposed 
in the abstract setting in \cite{SS}. The idea and proof of transport estimate based on this deformation 
argument are easy to state and prove. The difficulty is displaced into carrying on a concrete construction for each specific problem. 
\end{myindentpar}
\h The rest of the paper is organized as follows. In Section \ref{grsec}, we interpret
the Ohta-Kawasaki and nonlocal Mullins-Sekerka equations as gradient flows and introduce necessary notations and function spaces. In Section 
\ref{flowsec}, we briefly recall the $\Gamma$-convergence of gradient flows scheme in \cite{SS} and its particularization to our problem.
Then we prove a main inequality \`{a} la De Giorgi in Section \ref{mainlemsec} that will 
be crucial in the proof of Theorem \ref{dynamics}. We will present the proof of Theorem \ref{conjCH}
in Section \ref{CHpr}. Section \ref{pfsec} is devoted to the proof of Theorem \ref{dynamics}. The proof of Theorem
\ref{sdynamics} will be carried out in Section \ref{sthmsec}. In the final section, Section \ref{transsec}, we will prove Theorem  \ref{transport}.\\
\h {\it Note on constants and notations.} In this paper, we denote by $M$ a universal upper bound for the energy
of the initial data $E_{\e}(u^{\e}_{0})\leq M$ and $C$ a generic constant that may change from line to line but does not depend on $\e$.
For any function $f$ of space time variables $(x,t)$, we will write $f(t)$ for $f(\cdot, t).$\\
{\it Acknowledgements.} The author would like to thank Professor Sylvia Serfaty for her useful comments and suggestions and for communicating
 the proof of Theorem \ref{Schatzle} during the 
preparation of this paper. I am grateful to Professor Mark A. Peletier for his constructive comments and interesting discussion 
on an earlier version of 
the article.
The author is very grateful to the referees for 
their careful reading, useful comments and sharp critisms which resulted in a hopefully improved version of the 
original manuscript.

\section{Ohta-Kawasaki and nonlocal Mullins-Sekerka as gradient flows}
\label{grsec}
\h In this section, we introduce some notations used throughout the paper. In Section \ref{grOK}, we derive the gradient flow of 
the Ohta-Kawasaki functional defined in (\ref{OKfn}) with respect to an appropriately defined $H^{-1}$ structure. In Section \ref{grNO}, we present derivations of the gradient flows of 
the nonlocal area functional $E(u)$ defined in (\ref{NOener}) with respect to different structures. These derivations allow us to interpret
(\ref{OKeq}) and (\ref{NO}) as gradient flows. See \cite{HNR} for a different approach in interpreting (\ref{NO})
as a gradient flow. \\
\h The notion of gradient flow alluded to in this paper should be understood as follows.  Let $F$ be a $C^{1}$ 
functional defined over $\mathcal{M}$, an open subset of an affine space associated 
to a Hilbert space $X$ with inner product $<\cdot>_{X}$. By the $C^{1}$ character of $F,$ we can 
define the differential $dF(u)$ of $F$ at $u\in \mathcal{M}$ and denote by  $\nabla_{X}F(u)$ the 
vector of $X$ that represents it. That is, for all $\varphi\in \mathcal{M}$, we have
\begin{equation*}
 \frac{d}{dt}\mid_{t=0} F(u + t\varphi) = dF(u)\varphi = <\nabla_{X}F(u),\varphi>_{X}.
\end{equation*}
The gradient flow of $F$ with respect to the structure $X$ is the evolution equation
\begin{equation*}
 \partial_{t} u = -\nabla_{X}F(u).
\end{equation*}
\subsection{The gradient flows of the Ohta-Kawasaki functional}
\label{grOK}
Recall from the Introduction that $H_{n}^{-1}(\Omega)$ is a Hilbert space with inner product 
\begin{equation}
<u,v>_{H_{n}^{-1}(\Omega)} = \int_{\Omega} \nabla(\Delta_{n}^{-1} u)\cdot\nabla(\Delta_{n}^{-1} v) dx~~~\forall~ u,v\in H_{n}^{-1}(\Omega).
\label{xlemma}
\end{equation}
The gradient of the functional $E_{\varepsilon}$ defined by (\ref{OKfn}) with respect to the structure $H_{n}^{-1}(\Omega)$ 
is 
\begin{equation}\nabla_{H_{n}^{-1}(\Omega)} E_{\varepsilon}(u) = -\Delta(-\varepsilon \Delta u + \varepsilon^{-1} f(u) + 
\lambda \Delta^{-1}(u-\overline{u_{\Omega}})).\label{Xvareq}\end{equation}
Therefore, equation (\ref{OKeq}) is the gradient flow of $E_{\e}$ with respect to the $H^{-1}_{n}(\Omega)$ structure. 
\subsection{The gradient flows of the nonlocal area functional}
\label{grNO}
Consider a subdomain $\Omega^{-}$ of $\Omega$ with smooth boundary $\Gamma$. 
Assume further that $\Gamma$ {\it is the union of a finite number of disjoint closed surfaces}. This is the case of the interface $\Gamma(t)$ in our Theorems. 
Denote by $\Omega^{+}$ the set $\Omega\backslash \overline{\Omega^{-}}.$
Let $H^{1/2}(\Gamma)$ be the space of traces on
 $\Gamma$ of $H^{1}(\Omega^{-})$ functions. For $f \in H^{1/2}(\Gamma),$ let $X(f)$ be the set of extensions of $f$ into $H^{1}(\Omega)$ functions over $\Omega$. Then there 
exists a unique function $\tilde{f}\in X(f)$ minimizing the Dirichlet functional 
$\displaystyle \int_{\Omega} \abs{\nabla u}^{2} dx$ over $X(f).$ The function $\tilde{f}$ satisfies 
\begin{equation}
\Delta \tilde{f} =0 ~\text{in}~ \Omega\backslash \Gamma,~ \tilde{f}=f ~~\text{on} ~~\Gamma, ~\text{and}~ \frac{\partial \tilde{f}}{\partial n}=0~~\text{on} ~~\partial \Omega.  
\label{chartilde}
\end{equation}
\h With this $\tilde{f},$ we let $\Delta_{\Gamma} (f)= - \left [\frac{\partial \tilde{f}}{\partial n}\right]_ {\Gamma}$ ({\it The reader
will have not failed to note that, with abuse of notation, $\Delta_{\Gamma}$ in our definition is not the Laplace-Beltrami operator of 
$\Gamma$} ). 
Then, in the sense of distributions 
\begin{equation}
\Delta \tilde{f} = \Delta_{\Gamma}(f)\delta_{\Gamma}.
\label{twoLap}
\end{equation}
Now, for $f, u, v\in H^{1/2}(\Gamma)$, define
\begin{equation}  \norm{f}_{H_{n}^{1/2}(\Gamma)}=\norm{\nabla\tilde{f}}_{L^{2}(\Omega)},
~<u,v>_{H_{n}^{1/2}(\Gamma)}= <\nabla \tilde{u}, \nabla \tilde{v}>_{L^{2}(\Omega)}\equiv -\int_{\Gamma}(\Delta_{\Gamma} u)v~d\mathcal{H}^{N-1}.\label{seminorm}\end{equation}  Observe that $\norm{f}_{H_{n}^{1/2}(\Gamma)}=0$ iff $f$ is a constant on $\Gamma.$ So we can define the equivalence relation $\sim$ in $H^{1/2}(\Gamma):$ $f_{1}\sim f_{2}$ iff $\norm{f_{1}-f_{2}}_{H_{n}^{1/2}(\Gamma)} = 0.$\\
{\bf Notation.} Let $H_{n}^{1/2}(\Gamma)$ be the quotient space $H^{1/2}(\Gamma)/\sim.$ \\
\h Then, $H_{n}^{1/2}(\Gamma)$ with inner product $<\cdot,\cdot>_{H_{n}^{1/2}(\Gamma)}\label{mainlem}$ is a Hilbert space. Let 
$H_{n}^{-1/2}(\Gamma)$ be the dual of $H_{n}^{1/2}(\Gamma)$ with the usual dual norm $\norm{\cdot}_{H_{n}^{-1/2}(\Gamma)}.$ Then, we have 
\begin{lemma} (\cite{LeCH})
 (i) For each $u\in H_{n}^{-1/2}(\Gamma),$ there exists a unique $u^{\ast}\in H_{n}^{1/2}(\Gamma)$, denoted $\Delta_{\Gamma}^{-1} u$, 
such that $u = \Delta_{\Gamma} u^{\ast}$ and $\norm{u}_{H_{n}^{-1/2}(\Gamma)}=\norm{u^{\ast}}_{H_{n}^{1/2}(\Gamma)}$.
Moreover, for all $v\in H_{n}^{1/2}(\Gamma),$ $$\langle u,v\rangle_{H_{n}^{-1/2}(\Gamma)\times H_{n}^{1/2}(\Gamma)} = -<u^{\ast},v>_{H_{n}^{1/2}(\Gamma)}.$$
(ii) $H_{n}^{-1/2}(\Gamma)$ is a Hilbert space with inner product 
\begin{equation*}<u,v>_{H_{n}^{-1/2}(\Gamma)} = <\Delta_{\Gamma}^{-1} u, 
\Delta_{\Gamma}^{-1} v>_{H_{n}^{1/2}(\Gamma)}~~~\forall u,v\in H_{n}^{-1/2}(\Gamma). 
\end{equation*}
\label{downhalf}\end{lemma}
\h Now, for any $u\in BV(\Omega, \{-1,1\})$ with the interface $\Gamma =\partial\{x\in \Omega: u(x) = 1\}\cap \Omega$,
let $E(\Gamma)$ be the nonlocal area functional defined in (\ref{NOener}), which arises 
as the $\Gamma$-limit of the Ohta-Kawasaki functional $E_{\e}$. Denote by $v = \Delta^{-1}(u -\overline{u_{\Omega}}).$ \\
\h Then, with the choice of $\norm{\cdot}^{2}_{Y} = 4 \norm{\cdot}^{2}_{H^{-1/2}_{n}(\Gamma)}$, we have
\begin{propo}
Assume that $\Gamma$ is $C^{3}.$ Then the gradient of $E$ with respect to the structure $Y$ at 
$\Gamma$ is $ \nabla_{Y} E(\Gamma) = \frac{1}{2} \Delta_{\Gamma}(\sigma\kappa -\lambda v)\stackrel{\rightarrow}{n}$, where $\kappa$ is the 
mean curvature and $\stackrel{\rightarrow}{n}$ the unit outernormal vector to $\Gamma$. So if $\Gamma (t)$ is $C^{3}$ in 
space-time then the gradient flow of $E$ with respect to the structure $Y(t)$ at $\Gamma(t)$ is the nonlocal Mullins-Sekerka law (\ref{NO}).
\end{propo}
\begin{proof} 
Because $\Gamma$ is $C^{3},$ $\kappa$ is $C^{1}$ on $\Gamma$ and thus $\kappa\in H^{1/2}(\Gamma).$ Consider 
a smooth volume preserving deformation $\Gamma(t)$ of $\Gamma$ and let $V = (\partial_{t}\Gamma)\stackrel{\rightarrow}{n}$ be its normal 
velocity vector at $t=0$. The volume preserving condition implies that
\begin{equation}
 \int_{\Gamma}\partial_{t}\Gamma d\mathcal{H}^{N-1} =0 
\label{volpreserv}
\end{equation}
and the first variation formula gives
\begin{equation}\left.\frac{d}{dt}\right\rvert_{t=0} E(\Gamma(t)) = 
- 2<{\bf K},V>_{L^{2}(\Gamma)}\label{curvature}\end{equation} 
where ${\bf K} = (\sigma\kappa-\lambda v) \stackrel{\rightarrow}{n}$. This formula can be found in \cite{CS}; see formula (2.47) in 
the proof of Theorem 2.3 and Remark 2.8. For completeness, we 
indicate a simple derivation using only (\ref{volpreserv}). {\it This derivation will 
be used later in the proof of the construction of the deformation in Proposition \ref{deform}}. 
Let $\Omega^{-}(t)$ be the region enclosed by $\Gamma(t)$ and $\Omega^{+}(t) = \Omega\backslash \overline{\Omega^{-}(t)}$. 
Set $u(x,t) = 2\chi_{\Omega^{+}(t)}(x) -1$. Then $\left.\frac{d}{dt}\right\rvert_{t=0} u(x,t) =2\delta_{\Gamma}(x)\partial_{t}\Gamma(x)$.  Recall that 
\begin{equation}
E(\Gamma(t)) = \sigma\int_{\Omega} \abs{\nabla u(t)} + 
\frac{\lambda}{2}\int_{\Omega}\abs{\nabla v(t)}^2 dx
\label{NOenerNO}
\end{equation}
where $v(t) =\Delta^{-1}(u(t)-\overline{u(t)}_{\Omega}) .$ It is well-known that
\begin{equation}
\left.\frac{d}{dt}\right\rvert_{t=0} \sigma\int_{\Omega} \abs{\nabla u(t)} = 
\left.\frac{d}{dt}\right\rvert_{t=0} 2\sigma \mathcal{H}^{N-1}(\Gamma(t)) = -2\sigma<\kappa,\partial_{t}\Gamma>_{L^{2}(\Gamma)}.
\label{firstterm}
\end{equation}
For the variation of the second term on the left hand side of (\ref{NOenerNO}), we note that
\begin{equation*}
v(x,t) = \int_{\Omega}G(x,y)(u(y,t)-\overline{u_{\Omega}(t)}) dy + C(t)
\end{equation*}
for some constant $C(t)$, where $G$ is the Green's function of the operator $-\Delta$ on $\Omega$ with Neumann boundary condition. Integrating by parts gives
\begin{multline}
\frac{1}{2}\int_{\Omega}\abs{\nabla v(t)}^2 dx = \frac{1}{2}\int_{\Omega}-\Delta v(t) v(t) dx\\=\frac{1}{2}
\int_{\Omega}\int_{\Omega} G(x,y)(u(x,t)-\overline{u_{\Omega}(t)})(u(y,t)-\overline{u_{\Omega}(t)})dy dx.
\label{gradveq}
\end{multline}
By (\ref{volpreserv}), 
\begin{equation*}
\left.\frac{d}{dt}\right\rvert_{t=0} \overline{u(t)_{\Omega}} = \frac{1}{\abs{\Omega}}\int_{\Omega}2\delta_{\Gamma}\partial_{t}\Gamma dx 
=\frac{2}{\abs{\Omega}}\int_{\Gamma}\partial_{t}\Gamma d\mathcal{H}^{N-1} =0.
\end{equation*}
Hence, differentiating (\ref{gradveq}), we obtain
\begin{multline}
\left.\frac{d}{dt}\right\rvert_{t=0}\frac{1}{2}\int_{\Omega}\abs{\nabla v(t)}^2 dx=
\int_{\Omega}\int_{\Omega} G(x,y)(u(y,0)-\overline{u_{\Omega}(0)})\left(\left.\frac{d}{dt}\right\rvert_{t=0} u(x,t)\right) dydx\\=
\int_{\Omega} (v(x,0) - C(0))2\delta_{\Gamma}\partial_{t}\Gamma dx = 2<v,\partial_{t}\Gamma>_{L^{2}(\Gamma)}.
\label{secondterm}
\end{multline}
Combining (\ref{firstterm}) and (\ref{secondterm}), we get (\ref{curvature}).\\
Therefore, the gradient of $E$ with respect 
to the structure $L^{2}(\Gamma)$ at $\Gamma$ is 
\begin{equation}\nabla_{L^{2}(\Gamma)}E(\Gamma)= - 2{\bf K} = - 2(\sigma\kappa -\lambda v) \stackrel{\rightarrow}{n}.\label{Heq}\end{equation} 

Now, we calculate the $H_{n}^{-1/2}$- gradient $\nabla_{H_{n}^{-1/2}(\Gamma)} E(\Gamma) = D\stackrel{\rightarrow}{n}$ of $E(\Gamma)$ with 
respect to $H^{-1/2}_{n}(\Gamma)$. To do this, it suffices to express the quantity $\left.\frac{d}{dt}\right\rvert_{t=0}E(\Gamma(t))$ as 
an inner product in $H_{n}^{-1/2}(\Gamma)$: $\left.\frac{d}{dt}\right\rvert_{t=0}E(\Gamma(t))=<D,\partial_{t}\Gamma>_{H_{n}^{-1/2}(\Gamma)}.$
By Lemma \ref{downhalf}, and (\ref{seminorm}), we have \begin{eqnarray*}<D,\partial_{t}\Gamma>_{H_{n}^{-1/2}(\Gamma)}
 = <\Delta_{\Gamma}^{-1}D,\Delta_{\Gamma}^{-1} \partial_{t}\Gamma>_{H_{n}^{1/2}(\Gamma)}&=& 
-\int_{\Gamma}(\Delta_{\Gamma}^{-1}D)\cdot \Delta_{\Gamma}(\Delta_{\Gamma}^{-1}\partial_{t}\Gamma)d\mathcal{H}^{N-1}\\& =& 
-\int_{\Gamma}(\Delta_{\Gamma}^{-1}D)\cdot \partial_{t}\Gamma d\mathcal{H}^{N-1}.\end{eqnarray*} 
It follows from (\ref{Heq}) that $\Delta_{\Gamma}^{-1}D = 2(\sigma\kappa -\lambda v).$ In other words, 
the $H_{n}^{-1/2}$- gradient $\nabla_{H_{n}^{-1/2}(\Gamma)} E(\Gamma) $ of $E$ at $\Gamma$ is 
given by $\nabla_{H_{n}^{-1/2}(\Gamma)} E(\Gamma)= D\stackrel{\rightarrow}{n} = \Delta_{\Gamma}(2(\sigma\kappa -\lambda v))\stackrel{\rightarrow}{n}. $ 
Recalling $\norm{\cdot}^{2}_{Y} = 4 \norm{\cdot}^{2}_{H^{-1/2}_{n}(\Gamma)}$, we find 
that \begin{equation}\nabla_{Y} E(\Gamma) = \frac{1}{4} \nabla_{H_{n}^{-1/2}(\Gamma)} E(\Gamma) 
=\frac{1}{2} \Delta_{\Gamma}(\sigma\kappa -\lambda v)\stackrel{\rightarrow}{n} \label{yeq}\end{equation} and thus
the gradient flow of $E(\Gamma)$ with respect to 
the structure $Y$ at $\Gamma$ is $V = - \nabla_{Y} E(\Gamma) = - \frac{1}{2} \Delta_{\Gamma}(\sigma\kappa -\lambda v)\stackrel{\rightarrow}{n}.$ Recall 
the definition of $\Delta_{\Gamma}$ to find 
that $\partial_{t}\Gamma  = \frac{1}{2}\left[\frac{\partial\widetilde{(\sigma\kappa-\lambda v)}}{\partial n}\right]_{\Gamma}$ and this is
 equivalent to the nonlocal Mullins-Sekerka law (\ref{NO}). 
\end{proof}
\section{Gamma-convergence of gradient flows and key inequalities}
\label{flowsec}
In this section we briefly recall the $\Gamma$-convergence of gradient flows scheme in \cite{SS} and discuss how to apply this scheme to prove
the convergence of (\ref{OKeq}) to (\ref{NO}).
\subsection{General framework} 
\label{generalfr} First, we recall from \cite{SS} the following general strategy. \\
\h If $E_{\varepsilon}$ $\Gamma$-converges to $E$, then the key conditions for which the gradient 
flow of $E_{\varepsilon}$ with respect to the structure $X_{\varepsilon}$ $\Gamma$- converges to the gradient flow of $E$ with 
respect to the structure $Y$ are the following inequalities for general functions $u^{\varepsilon}$, not necessarily solving 
$\partial_{t} u^{\e} = -\nabla_{X_{\e}} E_{\e}(u^{\e}).$\\
(C1) ({\it Lower bound on the velocity}) For a subsequence such that $u^{\varepsilon}(t) \stackrel{S}{\longrightarrow} u(t)$, we 
have $u\in H^{1}((0,T), Y)$ and for every $s\in[0,T),$ $\liminf_{\varepsilon\rightarrow 0}\int_{0}^{s}
\norm{\partial_{t}u^{\varepsilon}(t)}_{X_{\varepsilon}}^{2}dt\geq \int_{0}^{s}\norm{\partial_{t}u(t)}^2_{Y}dt.$\\
(C2) ({\it Lower bound on the slope})~ If $u^{\varepsilon}\stackrel{S }{\longrightarrow} u$ then  
$\liminf_{\varepsilon\rightarrow 0}\norm{\nabla_{X_{\varepsilon}}E_{\varepsilon}(u^{\varepsilon})}_{X_{\varepsilon}}^{2}
\geq\norm{\nabla_{Y}E(u)}^2_{Y}.$\\
In the above conditions, $(S)$ is a sense of convergence to be specified in each problem. 
\subsection{The case of the Ohta-Kawasaki functional} Let us now particularize the above framework to (\ref{OKeq}) and (\ref{NO}). In our case, the sense $(S)$ is understood as
$L^{2}(\Omega)$ convergence and the functionals $E_{\e}$ and $E$ are defined by (\ref{OKfn})
and (\ref{NOener}), respectively. The space $X_{\e}$ and $Y$ are respectively $X_{\varepsilon} = H^{-1}_{n}(\Omega)$ and 
\begin{equation}\norm{\cdot}_{Y}^{2} = 4\norm{\cdot}^{2}_{H^{-1/2}_{n}(\Gamma)}.\label{Yspace}\end{equation}
By the results of Section \ref{grsec}, we are in the framework of the general scheme in \cite{SS}. \\
\h The first criterion $(C1)$ in the scheme now becomes
\begin{propo}
Let $u^{\varepsilon}$ be defined over $\Omega\times [0, T]$ such that 
$\int_{\Omega} \abs{u^{\varepsilon}(t)}^2 dx\leq M<\infty$ for all $t\in [0, T]$ and all $\varepsilon >0.$ Assume that, 
after extraction, $u^{\varepsilon}(t)\rightarrow u(t)$ in $L^{2}(\Omega)$ for all $t\in [0, T]$ where 
$u(t)\in BV(\Omega, \{-1, 1\})$ with interface $\Gamma(t) = \partial\{x\in\Omega: u(x, t)=1\}\cap\Omega.$ Then, for all 
$t\in (0, T),$ we have \begin{equation}\liminf_{\varepsilon\rightarrow 0}\int_{0}^{t} 
\norm{\partial_{t}u^{\varepsilon}(s)}_{H_{n}^{-1}(\Omega)}^{2} ds\geq \int_{0}^{t}\norm{\partial_{t}u(s)}_{H_{n}^{-1}(\Omega)}^{2} = 
4 \int_{0}^{t}\norm{\delta_{\Gamma(s)}\partial_{t}\Gamma(s)}_{H_{n}^{-1}(\Omega)}^{2} ds.\label{boundvelo}\end{equation}
\label{velolemma}\end{propo}
The proof of this Proposition is identical to that of Proposition 1.1 in \cite{LeCH}.\\
\h The second criterion $(C2)$ is equivalent to the following inequality \`{a} la De Giorgi: if $u^{\e}$ converges 
strongly in $L^{2}(\Omega)$ to $u\in BV(\Omega, \{-1,1\})$ with 
interface $\Gamma =\partial\{x\in\Omega: u(x) =1\}\cap\Omega$ then 
\begin{equation}
 \liminf_{\varepsilon\rightarrow 0}\int_{\Omega}
\abs{\nabla w^{\varepsilon}}^2 dx \geq  \norm{\sigma\kappa -\lambda v}^{2}_{H_{n}^{1/2}(\Gamma)}.
\label{motiineq}
\end{equation}
Here $w^{\e} = \varepsilon \Delta u^{\varepsilon}-\varepsilon^{-1}f(u^{\varepsilon}) -\lambda v^{\e}$ and $v^{\e} = \Delta^{-1}
(u^{\e}-\overline{u^{\e}_{\Omega}})$; $\kappa$ is the mean curvature of 
$\Gamma$  and $v =  \Delta^{-1}
(u-\overline{u_{\Omega}}).$
Indeed, from (\ref{xlemma}) and 
(\ref{Xvareq}), one can calculate
\begin{equation*}
 \norm{\nabla_{H_{n}^{-1}(\Omega)}E_{\varepsilon}(u^{\varepsilon})}_{H_{n}^{-1}(\Omega)}^{2} = 
\norm{\Delta w^{\e}}_{H_{n}^{-1}(\Omega)}^{2} =
\norm{\nabla w^{\varepsilon}}^{2}_{L^{2}(\Omega)}.
\end{equation*}
On the other hand, from (\ref{yeq}) and Lemma \ref{downhalf} $(ii)$, one deduces that
\begin{equation*}
 \norm{\nabla_{Y}E(\Gamma)}^2_{Y} = \norm{\frac{1}{2}\Delta_{\Gamma}(\sigma\kappa- \lambda v)}^2_{Y} = 
\norm{\Delta_{\Gamma}(\sigma\kappa-\lambda v)}^2_{H_{\color{red} n}^{-1/2}(\Gamma)} = \norm{\sigma\kappa-\lambda v}^2_{H_{\color{red} n}^{1/2}(\Gamma)}.
\end{equation*}
We will prove (\ref{motiineq}) in Lemma \ref{mainlemma} in Section \ref{mainlemsec}.
\subsection{Time-dependent limiting space} Let us emphasize that in \cite{SS}, the limiting space $Y$ is fixed. Assuming the validity of 
$(C1)$ and $(C2)$, the proof of the convergence of the gradient 
flow of $E_{\varepsilon}$ with respect to the structure $X_{\varepsilon}$ to the gradient flow of $E$ with 
respect to the structure $Y$ is quite short. In our case, we will apply $(C2)$ ( and (\ref{motiineq})) to $u^{\e}(t)$ where $u^{\e}$ is the solution 
of (\ref{OKeq}). Thus, $Y$ is time-dependent and it is not entirely clear how to carry out the scheme in 
\cite{SS}. Let us say right away that we just formally follow \cite{SS} and the time-dependent nature of $Y$ in our case is very special. The
most crucial point is that the term $\norm{\sigma\kappa -\lambda v}^{2}_{H_{n}^{1/2}(\Gamma)}$ on the left hand side of (\ref{motiineq})
can be expressed by a quantity defined globally on the whole domain $\Omega$. Precisely, we have
\begin{equation}
 \norm{\sigma\kappa -\lambda v}^{2}_{H_{n}^{1/2}(\Gamma)} = \inf_{\omega\in H^{1}({\Omega}),~ \omega=\sigma\kappa-\lambda v~on~\Gamma}\int_{\Omega} 
\abs{\nabla \omega}^2 dx.
\label{globalq}
\end{equation}
For each time slice $t$, (\ref{motiineq}) is a static statement. When considering the dynamics of (\ref{OKeq}), we use the function $\omega (x,t)$ such that
for each time slice $t$, $\omega(\cdot, t)\in H^{1}(\Omega)$ and 
realizes the infimum in (\ref{globalq}) for the quantity $\norm{\sigma\kappa -\lambda v}^{2}_{H_{n}^{1/2}(\Gamma(t))}$. The smoothness assumption
$(A2)$ on the time-track interface $\displaystyle \cup_{0\leq t\leq T}(\Gamma(t)\times \{t\})$ allows us to connect the values of $\omega$ on different
time slices. See, e.g, (\ref{Cauchy}) and (\ref{smoothness}) in the proof of Theorem \ref{dynamics}. Thus
the suspicion of the applicability of the scheme in \cite{SS} to our problem with the time-dependence nature of $Y$ can be more or less lifted in our proofs.
\section{An Inequality \`{a} la De Giorgi}
\label{mainlemsec}
\h In this section, we prove a main technical result, Lemma \ref{mainlemma}, that turns out to be crucial in the proof of 
Theorem \ref{dynamics}.\\
{\it Notes on notations.} In this section, we consider the smooth solution $(u^{\varepsilon}, v^{\varepsilon}, w^{\e})$ 
of (\ref{OKeq}) on $\Omega\times [0, \infty)$ with well-prepared initial data $u^{\e}_{0}$. By Proposition \ref{selection}, 
we can actually choose a subsequence of $\e$ such that $u^{\e}(\cdot,t)$ converges to 
$u^{0}(\cdot,t)\in BV(\Omega, \{-1,1\})$ in $L^{2}(\Omega)$ for all time slice 
$t$. For ease of notation, we drop the superscript $0$ in $u^{0}$. 
Denote $\Gamma(t) = \partial\{x\in\Omega: u^{0}(x,t)=1\}\cap\Omega$ and $\kappa(t)$ its mean curvature. Note that, due to the mass-preserving
nature of (\ref{OKeq}), we have for all $t\in [0,\infty)$
\begin{equation}
 \overline{u^{\e}_{\Omega}}(t) = \overline{u^{\e}_{0\Omega}} = m_{\e} \in (-m, m) ~~(0<m<1).
\label{massp}
\end{equation}
As always, we denote $\Delta^{-1}(u^{0}(s) -\overline{u^{0}_{\Omega}}(s))$ by $v(s)$. It is easy to see that $v^{\e}(t)\rightarrow v(t)$ in $H^{1}(\Omega)$ for each $t$.\\
\h Our main technical lemma reads
\begin{lemma}(Main Lemma)
Assume the time-track interface $\displaystyle \cup_{0\leq t\leq T}(\Gamma(t)\times \{t\})$ is $C^{3,\alpha}$. Then, there 
exists a positive constant $\delta(0)>0$ depending only on the initial data $u(0)$ such that for
$L^{1}$ a.e time slice $t\in [0, \delta(0)]$ we have \begin{eqnarray}\liminf_{\varepsilon\rightarrow 0}\int_{\Omega}
\abs{\nabla w^{\varepsilon}(t)}^2 dx\geq  \norm{\sigma\kappa(t) -\lambda v(t)}^{2}_{H_{n}^{1/2}(\Gamma(t))}.\label{dissipation}\end{eqnarray}
\label{mainlemma}
\end{lemma}
\begin{remark}
This is a nonlocal variant of an $H^{1}$-version of De Giorgi's conjecture \cite{DeG}. For more information on De Giorgi's conjectures and 
inequalities, we refer the reader to \cite{LeCH}. As explained by the end of the introduction and in Remark \ref{DGrk}, a static 
statement similar to (\ref{dissipation}) may be false. However, when $\lambda =0$, we have a purely static result as in Conjecture {(\bf{CH})}
and Theorem \ref{conjCH}.
\end{remark}

The rest of this section is devoted to proving Lemma \ref{mainlemma}. The proof of this lemma relies on 
\begin{myindentpar}{0.6cm}
(i) Tonegawa's  convergence theorem for diffused
interfaces whose chemical potentials are uniformly bounded in a Sobolev space;\\
(ii) R\"{o}ger's locality theorem for the weak mean curvature vector of an integral varifold;
and \\
(iii) Sch\"{a}tzle's constancy theorem for the density of an integral varifold with weak mean curvature in $L^{1}$.
\end{myindentpar}
\h First, we recall
\begin{theorem} (Tonegawa's  Convergence Theorem, Theorem 1 in \cite{TonePhase})
Suppose $p>\frac{N}{2}$ and let $\left\{u^{\varepsilon}\right\}_{0<\varepsilon\leq 1}$ be a sequence of $W^{3, p}(\Omega)$ functions 
satisfying \\
(a) The energy bound
\begin{equation*}
 \int_{\Omega}\left(\frac{\varepsilon}{2}\abs{\nabla u^{\e}}^2 +\frac{1}{\varepsilon}W(u^{\e})\right)dx \leq M<\infty,\end{equation*}
(b) The following uniform bound on the chemical potentials $\varepsilon\Delta u^{\varepsilon}-
\varepsilon^{-1}f(u^{\varepsilon})$
\begin{equation*}
 \norm{\varepsilon\Delta u^{\varepsilon}-
\varepsilon^{-1}f(u^{\varepsilon})}_{W^{1,p}(\Omega)}\leq M.
\end{equation*}
Then, after extraction,\\
(i) \begin{equation*}
     u^{\e}\rightarrow u~\text{in}~ L^{2}(\Omega),~ u\in BV(\Omega, \{-1,1\}),
    \end{equation*}
(ii) \begin{equation*}
      \varepsilon\Delta u^{\varepsilon}-
\varepsilon^{-1}f(u^{\varepsilon})\rightharpoonup F,~\text{weakly in}~ W^{1,p}(\Omega),
     \end{equation*}
(iii) there exists a Radon measure $\mu$ on $\Omega$ such that, 
in the sense of Radon measures, 
\begin{equation*}\left(\frac{\varepsilon\abs{\nabla u^{\varepsilon}}^2}{2} +
\frac{W(u^{\varepsilon})}{\varepsilon}\right)dx  \rightharpoonup \mu.
\end{equation*}
(iv) Moreover, $(2\sigma)^{-1}\mu$ is $(N-1)$-integer-rectifiable varifold with $(N-1)$-dimensional density
\begin{equation*}
 \theta^{(n-1)} (\mu,\cdot) = \theta(\cdot)2\sigma
\end{equation*}
where $\theta(\cdot)$ is integer-valued. \\
(v)Furthermore, $\mu$ has weak mean curvature $\overrightarrow{H_{\mu}}
\in L^{\frac{2(N-1)}{N-2}}_{loc}(\mu)$ and 
\begin{equation*}
 \overrightarrow{H_{\mu}} = \frac{F}{\theta\sigma}\nu \in L^{\frac{2(N-1)}{N-2}}_{loc}(\mu).
\end{equation*}
which holds $\mu$-almost everywhere, where $\nu =\frac{\nabla u}{\abs{\nabla u}}$ on $\partial^{\ast}\{u=1\}$ and $\nu=0$ elsewhere.\\
(vi) For $\mathcal{H}^{N-1}$ a.e. $x\in \partial^{\ast}\{u=1\}$, $\theta(x)$ is an odd integer.
\label{Toneconvthm}
\end{theorem}
\begin{theorem}(R\"{o}ger's locality theorem, Proposition 3.1 in \cite{Roger}) Let $E\subset\Omega$ be a set of finite perimeter, i.e, $\chi_{E}\in BV(\Omega)$. Assume 
that there are two (N-1)-integer-rectifiable varifolds $\mu_{1},\mu_{2}$ on $\Omega$ such that for $i=1,2$, the following hold:\\
(a) \begin{equation*}
     \partial^{\ast} E\subset \text{supp}~ \mu_{i}
    \end{equation*}
(b) $\mu_{i}$ has locally bounded first variation with weak mean curvature vector $ \overrightarrow{H_{\mu_{i}}}$,\\
(c) \begin{equation*}
      \overrightarrow{H_{\mu_{i}}}\in L^{s}_{\text{loc}} (\mu_{i}),~ s>\text{max}~\{N-1,2\}.
    \end{equation*}
Then
\begin{equation*}
  \overrightarrow{H_{\mu_{1}}}\mid_{\partial^{\ast} E} =  \overrightarrow{H_{\mu_{2}}}\mid_{\partial^{\ast} E}
\end{equation*}
\label{Rogerlocality}
\end{theorem} 
The above theorem justifies the definition of the weak curvature of $\partial^{\ast} E$ if there is an $(N-1)$-integer-rectifiable varifold $\mu$ satisfying (a)-(c).\\
\h Finally, we state the following result due to
Reiner Sch\"{a}tzle \cite{Schatzle} whose proof was communicated to us by Sylvia Serfaty.
\begin{theorem} (Sch\"{a}tzle's Constancy Theorem)
Let $\mu=\theta \mathcal{H}^{n}\lfloor M$ be an integral $n$-varifold in the open set $\Omega\subset R^{n+m}$, $M\subset \Omega$ a 
connected $C^{1}$-n-manifold, $\theta: M\rightarrow N_{0}$ be $\mathcal{H}^{n}$-measurable with weak 
mean curvature $\overrightarrow{H}_{\mu}\in L^{1}(\mu)$, that is
\begin{equation}
 \int \mathrm{div}_{\mu} \eta d \mu = \int_{M} \mathrm{div}_{M}\eta \theta d\mathcal{H}^{n} = -\int <\overrightarrow {H}_{\mu}, \eta> d\mu~~
\forall \eta\in C^{1}_{0}(\Omega, \RR^{n+ m}).
\label{weakMC}
\end{equation}
Then $\theta$ is a constant: $\theta \equiv\theta_{0}\in N_{0}$. Here $N_{0}$ is the set of all nonnegative integers and $<\cdot>$ is the 
standard Euclidean inner product on $\RR^{n + m}.$
\label{Schatzle}
 \end{theorem}
\begin{proof}
We consider locally $C^{1}$-vector fields $\nu^{1}, \cdots, \nu^{m}$ on $M$, which are an orthonormal basis of 
the orthogonal complement $TM^{\perp}$ of the tangent bundle $TM$ in $T\RR^{n +m}$. For $x\in M$, we 
choose an orthonormal basis $\tau_{1},\cdots, \tau_{n}$ of the tangent space $T_{x}M$ of $M$ at $x$. We decompose 
$\eta\in C^{1}_{0}(\Omega, \RR^{n + m})$ into $
 \eta =\eta^{tan} + \eta^{\perp}, $
where $$\eta^{tan}(x)=\pi_{T_{x}M}(\eta(x))\in T_{x}M,~\h
\eta^{\perp}(x)=\pi_{T_{x}M^{\perp}}(\eta(x))=\sum_{j=1}^{m}<\nu^{j}, \eta(x)>\nu^{j}\in T_{x}M^{\perp}.
 $$
Here, we have denoted $\pi_{V}$ the orthogonal projection operator on the subspace $V$ of $\RR^{n +m}$.
In particular, $
 \eta^{tan}, \eta^{\perp}\in C^{1}_{0}(\Omega).$ Then, we have $
 \mathrm{div}_{M} \eta = \mathrm{div}_{M} \eta^{tan}  + \mathrm{div}_{M} \eta^{\perp}. $
Let $D$ be the standard differentiation operator on $\RR^{n +m}$ 
and $A_{M}$ the second fundamental form of $M$. Denote by $\overrightarrow{H}_{M}$ the weak mean curvature of $M$. Then
\begin{equation*}
\overrightarrow{H}_{M} = \sum_{i=1}^{n} A_{M}(\tau_{i}, \tau_{i}).
\end{equation*}
We have
\begin{eqnarray*}
 \mathrm{div}_{M} \eta^{\perp} &=& \sum_{i=1}^{n}<\tau_{i}, \nabla^{M}_{\tau_{i}}\eta^{\perp}> =\sum_{i=1}^{n}\sum_{j=1}^{m}
<\tau_{i}, D_{\tau_{i}}\left(<\nu^{j}, \eta(x)>\nu^{j}\right)> \\&=&
\sum_{i=1}^{n}\sum_{j=1}^{m}
<\nu^{j},\eta>< \tau_{i}, D_{\tau_{i}}\nu^{j}>
= -<\eta, \sum_{i=1}^{n} A_{M}(\tau_{i}, \tau_{i})>
= -<\eta, \overrightarrow{H}_{M}>.
\end{eqnarray*}
From (\ref{weakMC}), we can calculate
\begin{eqnarray*}
 -\int <\overrightarrow{H}_{\mu},\eta> d\mu = -\int_{M}<\overrightarrow{H}_{\mu},\eta> \theta d\mathcal{H}^{n}
 &=& \int_{M}\mathrm{div}_{M}\eta\theta d\mathcal{H}^{n}\\
&=& \int_{M}\mathrm{div}_{M}\eta^{tan}\theta d\mathcal{H}^{n} +  \int_{M}\mathrm{div}_{M}\eta^{\perp}\theta d\mathcal{H}^{n}\\
&=&  \int_{M}\mathrm{div}_{M}\eta^{tan}\theta d\mathcal{H}^{n} - \int_{M}<\overrightarrow{H}_{M},\eta> \theta d\mathcal{H}^{n}.
\end{eqnarray*}
Let us make some special choices of $\eta$. First, for $\eta =\eta^{\perp} \in TM^{\perp}$, we conclude 
that the projection $\overrightarrow{H}^{\perp}_{\mu}$ of $\overrightarrow{H}_{\mu} $ on 
$TM^{\perp}$ satisfies $\overrightarrow{H}^{\perp}_{\mu} = \overrightarrow{H}_{M}$. Since $\mu$ is integral, 
we get $\overrightarrow{H}_{\mu} \bot T\mu = TM$ by Theorem 5. 8 in Brakke \cite{Brakke} and conclude $
 \overrightarrow{H}_{\mu} = \overrightarrow{H}_{M}.$
Finally, if we choose $\eta$ such that $\eta = \eta^{tan}\in TM$ then
\begin{equation*}
 \int_{M}\mathrm{div}_{M}\eta^{tan}\theta d\mathcal{H}^{n} =0.
\end{equation*}
Calculating in local coordinates, this yields $\nabla_{M}\theta =0$ weakly. Hence $\theta \equiv \theta_{0}$ is constant, as $M$ 
is connected.
\end{proof}

\h From the liminf inequality of $\Gamma$-convergence, we know that, for all $t$
\begin{equation*}
 \liminf_{\e\rightarrow 0} E_{\e}(u^{\e}(t)) \geq E(u(t))\equiv \sigma\int_{\Omega} \abs{\nabla u(t)} +\frac{\lambda}{2}\int_{\Omega}\abs{\nabla v(t)}^2 dx.
\end{equation*}
Using Sch\"{a}tzle's constancy theorem and Tonegawa's convergence theorem, we will improve the above inequality in (\ref{ttime}) as follows
\begin{equation}
 \liminf_{\e\rightarrow 0} E_{\e}(u^{\e}(t)) \geq \theta_{0}(t)\sigma\int_{\Omega} \abs{\nabla u(t)} +
\frac{\lambda}{2}\int_{\Omega}\abs{\nabla v(t)}^2 dx,
\label{highermult1}
\end{equation}
where $\theta_{0}(t)$ is an odd integer. In order to establish the convergence of (\ref{OKeq})
to (\ref{NO}) using the $\Gamma$-convergence of gradient flows scheme, we must rule out the higher multiplicity (i.e., the case where 
$\theta_{0}(t)>1$) of the interface $\Gamma(t)$ for all $t$ (see Remark \ref{DGrk}). Therefore, it is natural to find an upper bound for the
left hand side of (\ref{highermult1}) to ensure, with possibly extra conditions, that $\theta_{0}(t) =1.$ \\
%\h The above discussion motivates the following
%\begin{definition}
%We say that the functions $u^{\e}(x,t)$ is well-prepared at the time slice $t$ if
%\begin{equation}
% \lim_{\e\rightarrow 0} E_{\e}(u^{\e}(t)) = E(u(t))\equiv \sigma\int_{\Omega} \abs{\nabla u(t)} %+\frac{1}{2}\int_{\Omega}\abs{\nabla v(t)}^2.
%\end{equation}
%We say that the functions $u^{\e}(x,t)$ is mildly well-prepared at the time slice $t$ if
%\begin{equation}
 %\lim_{\e\rightarrow 0} E_{\e}(u^{\e}(t)) < 2\sigma\int_{\Omega} \abs{\nabla u(t)} %+\frac{1}{2}\int_{\Omega}\abs{\nabla v(t)}^2.
%\end{equation} 
%\end{definition}
\h As a first step to rule out the higher multiplicity issue of the limiting interfaces $\Gamma(t)$, we will use 
Theorem \ref{Schatzle} to establish the following important result concerning (\ref{OKeq}).
% the dynamics of the Ohta-Kawasaki equation.
\begin{propo}
Suppose that for each $t\in [0, T]$, $\Gamma(t)$ is $C^{2}$ and 
that the interface $\Gamma(t)$ is $C^{\alpha}$ in time (cf. (A2) of Theorem \ref{dynamics}), i.e, 
\begin{equation}
 \abs{\int_{\Omega} \abs{\nabla u(t)}-\abs{\nabla u(s)}}\leq C\abs{t-s}^{\alpha}~~\text{for some} ~\alpha>0.
\label{utime}
\end{equation}
Then, there exists $\delta(0)>0$ depending only on the initial data $u(0)$ such that the well-preparedness of initial data 
guarantees  for $L^{1}$ a.e. $t\in (0, 
\delta(0)]$,
the interface $\Gamma(t)$ has multiplicity one. Precisely, there exists a Radon measure $\mu(t)$ on $\Omega$ such that, 
up to extracting a subsequence, we have the following convergence in the sense of Radon measures, 
\begin{equation*}\left(\frac{\varepsilon\abs{\nabla u^{\varepsilon}(t)}^2}{2} +
\frac{W(u^{\varepsilon}(t))}{\varepsilon}\right)dx  \rightharpoonup \mu(t).
\end{equation*}
Moreover, $\Gamma(t)\subset supp~\mu(t)$; $(2\sigma)^{-1}\mu(t)$ is $(N-1)$-integer-rectifiable varifold with $(N-1)$-dimensional density
\begin{equation*}
 \theta^{(N-1)} (\mu(t),\cdot) = \theta(t)(\cdot)2\sigma
\end{equation*}
and
\begin{equation}
 \theta(t)(\cdot) \equiv 1~\text{on}~\Gamma(t).
\label{milddens}
\end{equation}
\label{mild}
\end{propo}
Here we call that the $(N-1)$-dimensional density $\theta^{(N-1)} (\mu(t),x)$ alluded to above is defined as follows
\begin{equation*}
 \theta^{(N-1)} (\mu(t), x) =\lim_{r\rightarrow 0} \frac{\mu(t) (B(x, r))}{\omega_{N-1} r^{N-1}}
\end{equation*}
where $\omega_{N-1}$ is the volume of the unit ball in $\RR^{N-1}$.\\
\h The idea of the proof is very simple. H\"{o}lder continuous hypersurfaces can not change much length in a short time. If we have higher 
constant integer multiplicity at a later time then to some extent, we will have more energy in $E_{\e}$. But this is a contradiction because the energy is 
decreasing in time for (\ref{OKeq}). Key to our proof is the following inequality for $t$ sufficiently small
\begin{equation*}
 \limsup_{\e\rightarrow 0} E_{\e}(u^{\e}(t)) < 2\sigma\int_{\Omega} \abs{\nabla u(t)} +\frac{\lambda}{2}\int_{\Omega}\abs{\nabla v(t)}^2 dx.
\end{equation*}
\\
\h As a preparation for the proof, we prove the following time-continuity estimates for $u^{\e}$ in $L^{2}(\Omega)$ and $v^{\e}$ in $H^{1}(\Omega)$. 
\begin{lemma}
(i) For all $s, t\in [0, T]$
\begin{equation}
 \norm{u^{\e}(s)-u^{\e}(t)}_{L^{2}(\Omega)}\leq C\abs{t-s}^{1/8}.
\label{volholder}
\end{equation}
(ii).   For all $s, t\in [0, T]$
\begin{equation}
\abs{ \int_{\Omega} \left(\abs{\nabla v^{\e}(s)}^2 -\abs{\nabla v^{\e}(t)}^2\right) dx}\leq C\abs{t-s}^{1/8}.
\label{vtime}
\end{equation}
\end{lemma}
\begin{proof}
Item (i) can be proved similarly as in the proof of Lemma 3.2 in \cite{ChenCH}. Now we prove $(ii)$. 
We have
\begin{equation*}
 \abs{ \int_{\Omega} \left(\abs{\nabla v^{\e}(s)}^2 -\abs{\nabla v^{\e}(t)}^2\right) dx}\leq (\norm{\nabla v^{\e}(s)}_{L^{2}(\Omega)}+
 \norm{\nabla v^{\e}(t)}_{L^{2}(\Omega)}) \norm{\nabla (v^{\e}(s)-v^{\e}(t))}_{L^{2}(\Omega)}.
\end{equation*}
The standard estimate
\begin{equation}
 \norm{\nabla v^{\e}}_{L^{2}(\Omega)}\leq C\norm{u^{\e}-\overline{u^{\e}_{\Omega}}}_{L^{2}(\Omega)}
\label{standa}
\end{equation}
combined with (\ref{massp}) implies that
\begin{equation*}
 \norm{\nabla (v^{\e}(s)-v^{\e}(t))}_{L^{2}(\Omega)}\leq C\norm{u^{\e}(s)-u^{\e}(t)}_{L^{2}(\Omega)}
\end{equation*}
Recalling $(i)$, we obtain the desired inequality.
\end{proof}
\h Now, we are ready to prove Proposition \ref{mild}.
\begin{proof}[Proof of Proposition \ref{mild}] To simplify the proof of our Proposition, we can assume further that $\Gamma(t)$ consists of one closed, connected hypersurface. Our proof can be modified easily to cover the case $\Gamma(t)$ consists of finitely many closed, connected hypersurfaces as in Theorem \ref{dynamics}. For each time slice $t\in [0, T]$, we have
\begin{equation*}
 E_{\e} (u^{\e}(t)) =  E_{\e}(u^{\e}(0))-\int_{0}^{t}\norm{\nabla w^{\e}(s)}^2_{L^{2}(\Omega)}ds\leq M.
\end{equation*}
In particular
\begin{equation}
 \int_{\Omega}\left(\frac{\varepsilon}{2}\abs{\nabla u^{\e}(t)}^2 +\frac{1}{\varepsilon}W(u^{\e}(t))\right) dx \leq E_{\e}(u^{\e}(t)) \leq M
\label{energyttime}
\end{equation}
and by Fatou's lemma, for $L^{1}$ a.e $t\in [0, T]$,
\begin{equation}
 \liminf_{\e\rightarrow 0} \norm{\nabla w^{\e}(t)}^2_{L^{2}(\Omega)}<\infty.
\label{Fatout}
\end{equation}
Recall that 
\begin{equation*}
 \varepsilon \Delta u^{\varepsilon}-\varepsilon^{-1}f(u^{\varepsilon})= w^{\e} + \lambda v^{\e}: = k^{\e}(t). 
\end{equation*}
From the energy bound and the mass constraint (\ref{massp}) and  in view of 
Lemma 3.4 in \cite{ChenCH}, which gives an upper bound for $\norm{k^{\varepsilon}(t)}_{H^{1}(\Omega)}$ in terms of the energy
 $E_{\varepsilon}(u^{\varepsilon}(t))$ and the homogeneous $H^{1}$-norm $\norm{\nabla k^{\varepsilon}(t)}_{L^{2}(\Omega)}$, we have for all $\varepsilon$ sufficiently 
small 
\begin{eqnarray*}
\norm{w^{\varepsilon}(t)+ \lambda v^{\e}(t)}_{H^{1}(\Omega)} = \norm{k^{\varepsilon}(t)}_{H^{1}(\Omega)} &\leq&
C(E_{\varepsilon}(u^{\varepsilon}(t)) + \norm{\nabla k^{\varepsilon}(t)}_{L^{2}(\Omega)} )\\&=&
 C (E_{\varepsilon}(u^{\varepsilon}(t)) + \norm{\nabla w^{\varepsilon}(t) + \lambda \nabla v^{\e}(t)}_{L^{2}(\Omega)})\\ &\leq& C (M + 
\norm{\nabla w^{\varepsilon}(t)}_{L^{2}(\Omega)} + \norm{\nabla v^{\e}(t)}_{L^{2}(\Omega)}).
\end{eqnarray*}
Moreover, (\ref{energyttime}) gives a uniform upper bound for $u^{\e}(t)$ in $L^{4}(\Omega)$ and hence
\begin{equation*}
\norm{u^{\e}(t)-\overline{u^{\e}_{\Omega}}(t)}_{L^{2}(\Omega)} \leq CM.
\end{equation*}
Because $v^{\e}(t)$ has average $\overline{v^{\e}_{\Omega}} =0$ for each $t$, the Poincare inequality and (\ref{standa}) gives
\begin{eqnarray}
 \norm{w^{\varepsilon}(t)}_{H^{1}(\Omega)}&\leq& \norm{w^{\varepsilon}(t)+ \lambda v^{\e}(t)}_{H^{1}(\Omega)} + \norm{-\lambda v^{\e}(t)}_{H^{1}(\Omega)}\nonumber\\ &\leq&
C (M + 
\norm{\nabla w^{\varepsilon}(t)}_{L^{2}(\Omega)} + \norm{u^{\e}(t)-\overline{u^{\e}_{\Omega}}(t)}_{L^{2}(\Omega)})\nonumber\\ &\leq& 
C(M + \norm{\nabla w^{\varepsilon}(t)}_{L^{2}(\Omega)} ).
\label{Chenlem}
\end{eqnarray}
and 
\begin{equation}
 \norm{k^{\varepsilon}(t)}_{H^{1}(\Omega)}= \norm{w^{\varepsilon}(t)+ \lambda v^{\e}(t)}_{H^{1}(\Omega)} \leq 
C(M + \norm{\nabla w^{\varepsilon}(t)}_{L^{2}(\Omega)} ).
\label{boundk}
\end{equation}
By (\ref{Fatout}), we have the uniform bound in $H^{1}(\Omega)$ of $k^{\e}(t)$ for a.e $t\in [0, T]$. 
This combined with (\ref{energyttime}) allows us to apply Tonegawa's convergence theorem (see Theorem 1 
in \cite{TonePhase}). {\it For ease of notation, we drop a.e for the moment.} Up to extracting a subsequence, $k^{\e}(t)$ converges weakly
 to $k(t)$ in $H^{1}(\Omega)$ and 
there exists a Radon measure $\mu(t)$ on $\Omega$ such that, 
in the sense of Radon measures, 
\begin{equation*}\left(\frac{\varepsilon\abs{\nabla u^{\varepsilon}}^2}{2} +
\frac{W(u^{\varepsilon})}{\varepsilon}\right)dx  \rightharpoonup \mu(t).
\end{equation*}
Moreover, $(2\sigma)^{-1}\mu(t)$ is $(N-1)$-integer-rectifiable varifold with $(N-1)$-dimensional density
\begin{equation}
 \theta^{(n-1)} (\mu(t),\cdot) = \theta(t)(\cdot)2\sigma
\end{equation}
where $\theta(t)(\cdot)$ is integer-valued. Furthermore, $\mu(t)$ has weak mean curvature $\overrightarrow{H_{\mu}} (t)
\in L^{\frac{2(N-1)}{N-2}}_{loc}(\mu)$ and 
\begin{equation}
 \overrightarrow{H_{\mu}}(t) = \frac{k(t)}{\theta(t)\sigma}\nu \in L^{2}(\mu(t)).
\label{MC}
\end{equation}
which holds $\mu$-almost everywhere, where $\nu =\frac{\nabla u}{\abs{\nabla u}}$ on $\partial^{\ast}\{u=1\}\cap 
\Omega=\Gamma(t)$ and $\nu=0$ elsewhere.\\
\h It follows from our assumption $N\leq 3$ that $
 \frac{2(N-1)}{N-2}> \text{max}\{N-1,2\}. $ Thus, the locality result of R\"{o}ger in Theorem \ref{Rogerlocality} applies. Because $\Gamma(t) \subset~\text{supp}\mu(t)$, we see that
$\theta(t): \Gamma(t)\rightarrow N_{0}$ is $\mathcal{H}^{N-1}$-measurable and $2\sigma\theta(t)\mathcal{H}^{ N-1}
\lfloor\Gamma(t)$ has weak mean curvature
\begin{equation}
 \overrightarrow{H_{\mu}}(t) = \frac{k(t)}{\theta(t)\sigma}\frac{\nabla u}{\abs{\nabla u}} \in L^{2}(2\sigma\theta(t)\mathcal{H}^{N-1}
\lfloor\Gamma(t)).
\label{weakmu}
\end{equation}
By Sch\"{a}tzle's Theorem, $
 \theta(t)(\cdot) $ is a constant  $\theta_{0}(t)$ on $\Gamma(t).$ Moreover, \cite{TonePhase} shows 
that $\theta_{0}(t)$ is an odd integer.\\
\h Now, the constancy of $\theta$ on $\Gamma(t)$ gives
\begin{multline}
 \liminf_{\e\rightarrow 0} E_{\e}(u^{\e}(t)) \geq 2\theta_{0}(t) \sigma\mathcal{H}^{N-1}(\Gamma(t)) + \frac{\lambda}{2}\int_{\Omega}\abs{\nabla v(t)}^2 dx
\\=
\theta_{0}(t)\sigma\int_{\Omega} \abs{\nabla u(t)} +\frac{\lambda}{2}\int_{\Omega}\abs{\nabla v(t)}^2 dx.
\label{ttime}
\end{multline}
Moreover, from the proof of Theorem \ref{Schatzle}, one has $\overrightarrow{H_{\mu}}(t) = \overrightarrow{H_{\Gamma(t)}}$
. Because $\Gamma(t)$ is $C^{2}$, by Corollary 4.3 in \cite{Schatzle1}, the weak mean curvature vector coincides with the classical 
mean curvature vector. Hence, (\ref{weakmu}) gives
\begin{equation}
\kappa(t) =  \frac{k(t)}{\theta_{0}(t)\sigma}.
 \label{local2}
\end{equation}
From (\ref{utime}) and (\ref{vtime}), we can estimate
\begin{multline*}
 2\sigma\int_{\Omega} \abs{\nabla u(t)} +\frac{\lambda}{2}\int_{\Omega}\abs{\nabla v(t)}^2 dx -
 \left(\sigma\int_{\Omega} \abs{\nabla u(s)} +\frac{\lambda}{2}\int_{\Omega}\abs{\nabla v(s)}^2 dx\right)\\ \geq -2C\sigma\abs{t-s}^{\alpha} 
-C\abs{t-s}^{1/8} + \sigma\int_{\Omega} \abs{\nabla u(s)}.
\end{multline*}
Thus, we can find $\delta = \delta (u^{0},s)>0$ depending only on the initial 
data and $s$ such that for all $t\in [s, s+\delta)$
\begin{equation}
 2\sigma\int_{\Omega} \abs{\nabla u(t)} +\frac{\lambda}{2}\int_{\Omega}\abs{\nabla v(t)}^2 dx >
 \sigma\int_{\Omega} \abs{\nabla u(s)} +\frac{\lambda}{2}\int_{\Omega}\abs{\nabla v(s)}^2 dx.
\label{highermult}
\end{equation}
Assuming we have the well-preparedness at time $s\geq 0$. Then
\begin{equation}
 \lim_{\e\rightarrow 0} E_{\e}(u^{\e}(s)) =\sigma\int_{\Omega} \abs{\nabla u(s)} +\frac{\lambda}{2}\int_{\Omega}\abs{\nabla v(s)}^2 dx.
\label{stime}
\end{equation}
Because the Ohta-Kawasaki functional is decreasing along the flow, one has for $t>s$
\begin{equation}
  \limsup_{\e\rightarrow 0} E_{\e}(u^{\e}(t))\leq  \lim_{\e\rightarrow 0} E_{\e}(u^{\e}(s)).
\label{sttime}
\end{equation}
Thus from (\ref{ttime}), (\ref{stime}) and (\ref{sttime}), one finds that, for $L^{1}$ a.e $t\in [s, T]$, 
\begin{equation}
 \theta_{0}(t)\sigma\int_{\Omega} \abs{\nabla u(t)} +\frac{\lambda}{2}\int_{\Omega}\abs{\nabla v(t)}^2 dx 
\leq \sigma\int_{\Omega} \abs{\nabla u(s)} +\frac{\lambda}{2}\int_{\Omega}\abs{\nabla v(s)}^2 dx.
\label{finaltime}
\end{equation}
Revoking (\ref{highermult}) and (\ref{finaltime}), we conclude that the interface $\Gamma(t)$ has 
single multiplicity $\theta_{0}(t) =1$ for $L^{1}$ a.e. $t\in [s, s+\delta)$, i.e., (\ref{milddens}) is satisfied.
Therefore, the proof of Proposition \ref{mild}
is complete by setting $s=0$.
\end{proof}
\begin{remark}
The inequality (\ref{ttime}) can only be strict in the presence of hidden boundary, i.e, the set $supp\mu(t)\backslash\Gamma(t)$ is not empty and 
has positive 
$(N-1)$-dimensional Hausdorff measure. The set  $supp\mu(t)\backslash\Gamma(t)$ is one where $\nu =0$ in (\ref{MC}).
\label{hidden}
\end{remark}
\begin{remark}
Our proof shows that well-preparedness of the data
at any time $s$ will ensure (\ref{milddens}) for all $t\in [s, s + \delta(s)]$  with single multiplicity for $\Gamma(t)$.  
\label{conti}
\end{remark}
\h Finally, we give the proof of Lemma \ref{mainlemma}.
\begin{proof}[Proof of Lemma \ref{mainlemma}]
Consider $t\in [0,\delta(0)]$ where $\delta(0)$ is defined in the proof of Proposition \ref{mild}. We can assume that 
$\liminf_{\e\rightarrow 0} \int_{\Omega}\abs{\nabla w^{\varepsilon}(t)}^2\leq C$, 
otherwise the inequality (\ref{dissipation}) is trivial. Let $k^{\e}(t) = w^{\e}(t) + \lambda v^{\e}(t)$. Recall from (\ref{Chenlem}) and (\ref{boundk}) that
\begin{equation}
 \norm{w^{\varepsilon}(t)}_{H^{1}(\Omega)} + \norm{k^{\varepsilon}(t)}_{H^{1}(\Omega)} \leq  
C(M + \norm{\nabla w^{\varepsilon}(t)}_{L^{2}(\Omega)})\leq C.
\label{wH1}
\end{equation}
Now, up to extraction, we have $w^{\varepsilon}(t)$ and $k^{\e}(t)$ weakly converge in $H^{1}(\Omega)$ to some 
$w(t)$ and $k(t)$, respectively. Inspecting the proof of 
Proposition \ref{mild}, one observes that well-preparedness of the initial data together with (\ref{wH1}) 
implies (\ref{milddens}) at the time slice $t$, that is, the interface $\Gamma(t)$ has constant multiplicity $\theta_{0}(t) =1$. Thus, from (\ref{local2}) with the 
constant $\theta \equiv 1$, one deduces $k(t) =\sigma\kappa(t)$ on $\Gamma(t)$. Letting $\e\rightarrow 0$ in $k^{\e}(t) = w^{\e}(t) + \lambda v^{\e}(t)$, 
one gets $k(t) = w(t) + \lambda v(t)$. Hence
$w(t) = \sigma\kappa(t)-\lambda v(t)$ on $\Gamma(t)$.\\ 
\h By lower semicontinuity, one has 
\begin{equation}\liminf_{\varepsilon\rightarrow 0}\int_{\Omega}\abs{\nabla w^{\varepsilon}(t)}^2 dx
\geq \int_{\Omega}\abs{\nabla w(t)}^{2} dx \geq \inf_{\omega\in H^{1}({\Omega}),~ \omega=\sigma\kappa-\lambda v~on~\Gamma(t)}\int_{\Omega} 
\abs{\nabla {\color{red}\omega}}^2 dx.\label{lsc}\end{equation} 
\h The latter minimization  problem has a unique solution 
$\omega=\widetilde{\sigma\kappa(t) -\lambda v(t)}$ as defined in Section 2. 
Therefore, from (\ref{lsc}) and (\ref{seminorm}), we obtain 
\begin{equation}\liminf_{\varepsilon\rightarrow 0}\int_{\Omega}\abs{\nabla w^{\varepsilon}(t)}^2 dx
\geq \norm{\sigma\kappa(t)-\lambda v(t)}_{H^{1/2}_{n}(\Gamma(t))}^2.\label{withm}\end{equation}
\end{proof}
\begin{remark}
 It is very important to obtain the single multiplicity of the interface $\Gamma(t)$ in the proof of Lemma \ref{mainlemma}. In general, if $\Gamma(t)$ has
constant multiplicity $m$ then $k(t) = m\sigma\kappa(t)$ on $\Gamma(t)$ and the best inequality one can get is the following
\begin{equation*}\liminf_{\varepsilon\rightarrow 0}\int_{\Omega}\abs{\nabla w^{\varepsilon}(t)}^2 dx
\geq \norm{m\sigma\kappa(t)-\lambda v(t)}_{H^{1/2}_{n}(\Gamma(t))}^2\end{equation*}
where the quantity on the right hand side can be much smaller than the expected quantity $\norm{\sigma\kappa(t)-\lambda v}_{H^{1/2}_{n}(\Gamma(t))}^2$. This is
in contrast to an $H^{1}$-version of De Giorgi's conjecture (see Theorem 1.2 in \cite{LeCH} and Theorem
\ref{conjCH} in this paper) where any constant multiplicity suffices the proof.
\label{DGrk}
\end{remark}
\section{Proof of Theorem \ref{conjCH}}
\label{CHpr}
In this section, we present the proof of Theorem \ref{conjCH}.
\begin{proof}[Proof of Theorem \ref{conjCH}]
Let $k^{\e} = \e\Delta u^{\e}-\e^{-1} f(u^{\e})$. We can assume that 
$\displaystyle \liminf_{\e\rightarrow 0}\int_{\Omega}\abs{\nabla k^{\varepsilon}}^2 dx\leq C$, otherwise the inequality (\ref{DGineq}) is trivial.
From the energy bound and the mass constraint (\ref{fixedmass}) and  in view of Lemma 3.4 in \cite{ChenCH}, we have for all 
$\varepsilon$ sufficiently small \begin{equation*}\norm{k^{\varepsilon}}_{H^{1}(\Omega)}\leq 
C \left(\int_{\Omega}\left(\frac{\varepsilon}{2}\abs{\nabla u^{\e}}^2 +\frac{1}{\varepsilon}W(u^{\e})\right) dx +
 \norm{\nabla k^{\varepsilon}}_{L^{2}(\Omega)}\right)\leq C<\infty.\label{Sobolev}\end{equation*} 
Now, up to extraction, we have that $k^{\varepsilon}$ weakly converges to some $k$ in $H^{1}(\Omega).$ As in the proof of 
Proposition \ref{mild}, especially following (\ref{boundk})-(\ref{local2}), we can find an odd integer $\theta_{0}$ such that
\begin{equation*} 
k= \theta_{0}\sigma\kappa ~\text{on}~ \Gamma ~\text{a.e}~ \mathcal{H}^{N-1}.
\end{equation*}
Now, by lower semicontinuity, one has 
\begin{equation}\liminf_{\varepsilon\rightarrow 0}\int_{\Omega}\abs{\nabla k^{\varepsilon}}^2 dx
\geq \int_{\Omega}\abs{\nabla k}^{2} dx\geq \inf_{w\in H^{1}({\Omega}),~ w=\theta_{0}\sigma\kappa~on~\Gamma}\int_{\Omega} \abs{\nabla w}^2 dx
=\theta_{0}^2\sigma^2\norm{\kappa}_{H^{1/2}_{n}(\Gamma)}^2.\label{lscnew}\end{equation} 
Because $\theta_{0}$ is an odd integer, $\abs{\theta_{0}}\geq 1$. This combined with (\ref{lscnew}) gives (\ref{DGineq}) as desired.
\end{proof}
\begin{remark}
 In view of a recent result by R\"{o}ger and Tonegawa \cite{RT}, we might expect $\theta_{0}$ to be exactly $1$.
\end{remark}

\section{Proof of Theorem \ref{dynamics}}
\label{pfsec}
\h In this section, we prove Theorem \ref{dynamics}, formally following \cite{SS} (see also \cite{LeCH} for related results for the Cahn-Hilliard 
equation). \\
\h First, we briefly discuss the selection result alluded to in Section \ref{mainlemsec}.\\
\h For the rest of the section, $(u^{\varepsilon}, v^{\varepsilon}, w^{\e})$ denotes 
the solution of (\ref{OKeq}) on $\Omega\times [0,\infty).$ Let $T>0$ be any finite number. We define the following norm on distributions $u$ on $\Omega$ 
\begin{equation}
 \norm{u}_{1} = \sup_{\varphi \in C_{0}^{\infty}(\Omega),~\abs{\nabla\varphi}\leq 1}\abs{\int_{\Omega} u\varphi},
\label{1norm}\end{equation}
i.e., the norm in the dual of Lipschitz functions. Then, we have the following
\begin{propo}
There exists $u^{0}\in L^{4}(\Omega\times [0, T])$ such that $u^{0}$ is $C^{0,1/2}$ in time for the $\norm{\cdot}_{1}$-norm, and that, after extraction, \begin{equation} u^{\varepsilon}\rightharpoonup u^{0}~~\text{in}~ L^{4}(\Omega\times [0, T])\label{spacet}.\end{equation} Moreover, for all $t\in [0, T]$, we have  $u^{0}(t)\in BV (\Omega,\{-1,1\})$ and
\begin{equation} u^{\varepsilon}(t)\rightharpoonup u^{0}(t)~~\text{in}~ L^{4}(\Omega),~ u^{\varepsilon}(t)\longrightarrow u^{0}(t)~~\text{in}~ L^{2}(\Omega) .
\label{spaceonly}\end{equation}
\label{selection}
\end{propo}
The proof of this Proposition is similar to that of Proposition 4.1 in \cite{LeCH} and is thus omitted.
\begin{remark}
 For each $t$, from the energy bound $E_{\e}(u^{\e}(t))\leq E_{\e}(u^{\e}(0))\leq M$ and the compactness of BV functions in $L^{1}(\Omega)$, we can obtain
(\ref{spaceonly}) for a subsequence of $\e$'s. In general, this subsequence depends on $t$. The main point of Proposition \ref{selection} is that
this subsequence can be chosen independent of $t$. This follows from the time-continuity of $u^{0}$ in the $\norm{\cdot}_{1}$-norm. See Proposition 4.1 in \cite{LeCH}
for more details.
\end{remark}
\h Now, we are in a position to present the proof of Theorem \ref{dynamics}.
\begin{proof}[Proof of Theorem \ref{dynamics}]
{\bf 1.} First,  we note that the nonlocal Mullins-Sekerka law (\ref{NO}) with smooth initial interface $\Gamma(0)$ has unique smooth solution \cite{EN}. 
Thus, if $T_{\ast}$ is the minimum of the collision time and of the exit time from $\Omega$ of the hypersurfaces under the motion law
(\ref{NO}), then $T_{\ast}>0$. By the selection 
result in Proposition \ref{selection}, after extraction, we have that for all $t\in [0, T_{\ast}]$, $u^{\varepsilon}(\cdot,t)$ converges 
strongly in $L^{2}(\Omega)$ to  $u^{0}(\cdot,t)\in$ BV $(\Omega,\{-1,1\})$ with 
interface $\Gamma(t) = \partial\{x\in\Omega: u^{0}(x,t)=1\}\cap\Omega.$ By our assumption (A2) on the regularity of the time-track interface 
$\displaystyle \cup_{0\leq t\leq T_{\ast}}(\Gamma(t)\times \{t\})$, Lemma \ref{mainlemma} can be applied. Choose $\delta(0)>0$ as in Lemma
\ref{mainlemma}. Without loss of generality, one can assume that 
$\delta(0)<T_{\ast}.$ We proceed as follows.
 First, we confirm the evolution law on $[0, \delta(0)]$. Then we can easily iterate to continue the dynamics up to time $T_{\ast}$.\\
\h Let us prove that the 
interfaces $\Gamma(t)$ $(t\in [0, \delta(0)])$ evolve by the nonlocal Mullins-Sekerka law (\ref{NO}). Indeed, we have $
 \partial_{t}u^{\e} = -\nabla_{H_{n}^{-1}(\Omega)}E_{\varepsilon}(u^{\varepsilon}) $
and, for all $t\in (0, \delta(0)]$, 
\begin{eqnarray*}
 E_{\varepsilon}(u^{\varepsilon}(0))-E_{\varepsilon}(u^{\varepsilon}(t))&=&
-\int_{0}^{t}<\nabla_{H_{n}^{-1}(\Omega)}E_{\varepsilon}(u^{\varepsilon}(s)),\partial_{t}u^{\varepsilon}(s)>_{H_{n}^{-1}(\Omega)}ds\\&=&
 \frac{1}{2}\int_{0}^{t} \left(\norm{\nabla_{H_{n}^{-1}(\Omega)}E_{\varepsilon}(u^{\varepsilon}(s))}_{H_{n}^{-1}(\Omega)}^{2} +
 \norm{\partial_{t}u^{\varepsilon}(s)}_{H_{n}^{-1}(\Omega)}^{2}\right)ds\\& =&
 \frac{1}{2}\int_{0}^{t} \left(\norm{\nabla w^{\varepsilon}(s)}^{2}_{L^{2}(\Omega)} + 
\norm{\partial_{t}u^{\varepsilon}(s)}_{H_{n}^{-1}(\Omega)}^{2}\right)ds.
 \end{eqnarray*}
For each $s\in (0,t)$, recall that $\kappa(s)$ is the mean curvature of $\Gamma(s)$. Let $w(\cdot,s)\in H^{1}(\Omega)$ 
be the function $\widetilde{\sigma\kappa(s)-\lambda v(s)}$, i.e., $w(\cdot,s)$ satisfies $\Delta w(\cdot,s)=0$ in 
$\Omega\setminus \Gamma(s),$ 
$w(\cdot,s)=\sigma\kappa(s) -\lambda v (s)$ on $\Gamma(s)$ and finally $\frac{\partial w}{\partial n}=0$ on $\partial\Omega.$ By Proposition 
\ref{selection}, all assumptions of Proposition \ref{velolemma} are satisfied for $u^{\varepsilon}$ and $u^{0}$. 
Thus, by Lemma \ref{mainlemma}, the lower bound on velocity (\ref{boundvelo}) and the Cauchy-Schwarz inequality, we obtain
 \begin{eqnarray}\displaystyle
    E_{\varepsilon}(u^{\varepsilon}(0))-E_{\varepsilon}(u^{\varepsilon}(t))&\geq& 
\frac{1}{2}\int_{0}^{t} \left(\norm{\sigma\kappa -\lambda v}_{H_{n}^{1/2}(\Gamma(s))}^{2} + 
4\norm{\delta_{\Gamma(s)}\partial_{t}\Gamma(s)}_{H_{n}^{-1}(\Omega)}^{2}\right)ds -o(1)\label{strong}
\\&=&\frac{1}{2}\int_{0}^{t}\int_{\Omega}\left(\abs{\nabla w(x,s)}^2 + 
4\abs{\nabla \Delta_{n}^{-1}(\delta_{\Gamma(s)}\partial_{t}\Gamma(x,s))}^2\right) dx ds-o(1)\nonumber\\&\geq& -2\int_{0}^{t}\int_{\Omega}
\nabla w(x,s)\cdot\nabla (\Delta^{-1}_{n}(\delta_{\Gamma(s)}\partial_{t}\Gamma(x,s))) dxds -o(1)\label{Cauchy}.\end{eqnarray}
In view of the definition of $\Delta_{n}^{-1}$ in (\ref{xlemma}), the right hand side of (\ref{Cauchy}) becomes
\begin{eqnarray}
  2\int_{0}^{t}<\partial_{t}\Gamma(s),w>_{L^{2}(\Gamma(s))}ds-o(1)&=& 
\int_{0}^{t}\int_{\Gamma(s)} 2(\sigma\kappa(s)-\lambda v)\partial_{t}\Gamma(s) d\mathcal{H}^{N-1} ds -o(1)\nonumber\\&=& 
-\int_{0}^{t}\frac{d}{ds} E(\Gamma(s))ds -o(1)= E(\Gamma(0))-E(\Gamma(t))-o(1)\label{smoothness}.\end{eqnarray}
Equality (\ref{smoothness}) follows from the smoothness assumption (A2). From (\ref{strong})-(\ref{smoothness}), one gets $$E_{\varepsilon}(u^{\varepsilon}(t))-E(\Gamma(t))\leq E_{\varepsilon}(u^{\varepsilon}(0))-E(\Gamma(0)) + o(1).$$
 By $(A1),$ we deduce that $\limsup_{\varepsilon\rightarrow 0}E_{\varepsilon}(u^{\varepsilon}(t))\leq E(\Gamma(t)) .$ However, since $E_{\varepsilon}$ $\Gamma-$ converges to $E$, we have $\liminf_{\varepsilon\rightarrow 0}E_{\varepsilon}(u^{\varepsilon}(t))\geq E(\Gamma(t)).$ 
 Therefore, we must have 
\begin{equation}\lim_{\varepsilon\rightarrow 0}E_{\varepsilon}(u^{\varepsilon}(t))= E(\Gamma(t)).\label{singlet}\end{equation}
  This means that well-prepared initial data remains ``well-prepared'' in time for all $t\in [0,\delta(0)]$ and there are no hidden
boundaries in the limit measure of $E_{\e}(u^{\e}(t))$ (see Remark \ref{hidden}).
Furthermore, this also shows that the inequality (\ref{Cauchy}) is actually an equality. This implies that for each $s\in (0, t)$ and 
for a.e $x\in \Omega$, we have $\nabla w(x,s) = - 2\nabla \Delta_{n}^{-1}(\delta_{\Gamma(s)}\partial_{t}\Gamma(x,s)).$ So $w(x,s) = 
- 2\Delta_{n}^{-1} (\delta_{\Gamma(s)}\partial_{t}\Gamma(x,s)) +c(s)$ for some function $c$ depending only on time. Thus, in the sense 
of distributions $\delta_{\Gamma(s)}\partial_{t}\Gamma(x,s)= - \frac{1}{2}\Delta w(x,s )$.  By (\ref{twoLap}) and the definition of the function $w$,
 this relation is exactly the limiting dynamical law we wish to establish. Our proof of this nonlocal Mullins-Sekerka law is valid as 
long as $\Gamma(t)\subset\Omega$ and hypersurfaces contained in $\Gamma(t)$ do not collide for all $t< T_{\ast}$.\\
\h Now, starting from the time $\delta(0)$ with well-preparedness, we can use Remark \ref{conti} and Lemma \ref{mainlemma} to 
confirm the evolution law on $[\delta(0),\delta(1)]$ where 
$\delta(1) =\delta(\delta(0))$ defined in the proof of Proposition \ref{mild}. Define $\delta(k) =\delta(\delta(k-1))$. Due to the strict
positivity of the area $\int_{\Omega}\abs{\nabla u(t)}$ for any $t$, and from the construction of $\delta (k)$, we can show that
\begin{equation*}
 \lim_{k\rightarrow \infty}\delta(k) = T_{\ast}
\end{equation*}
where $T_{\ast}$ can be chosen to be the minimum of the collision time and of the exit 
time from $\Omega$ of the hypersurfaces under the nonlocal Mullins-Sekerka law.\\            
 {\bf 2.} Second, we show that $w^{\varepsilon}$ converges weakly in $L^{2}((0, T_{\ast}), H^{1}(\Omega))$ to $w.$ Indeed, 
for all $t\in (0, T_{\ast})$ we have 
\begin{eqnarray}
 \int_{0}^{t} \norm{\nabla w^{\varepsilon}(s)}^{2}_{L^{2}(\Omega)} ds =
 E_{\varepsilon}(u^{\varepsilon}(0))-E_{\varepsilon}(u^{\varepsilon}(t))\leq M.  
\label{Hminus1}
 \end{eqnarray} 
Recall from (\ref{Chenlem}) that    
\begin{equation*}
 \norm{w^{\varepsilon}(s)}_{H^{1}(\Omega)} \leq 
C(M + \norm{\nabla w^{\varepsilon}}_{L^{2}(\Omega)} + 1)\leq C.
\end{equation*}
It follows that for $\varepsilon$ sufficiently small, we have 
$$ \int_{0}^{t} \norm{ w^{\varepsilon}(s)}^{2}_{H^{1}(\Omega)} ds 
\leq C(M^2 +\int_{0}^{t} \norm{\nabla w^{\varepsilon}(s)}^{2}_{L^{2}(\Omega)} ds + \norm{u^{\e}}^2_{L^{2}(\Omega\times [0, T))}) \leq C< \infty.$$
Therefore, up to a further extraction, we have that 
$w^{\varepsilon}$ weakly converges to some $z$ in $L^{2}((0, T_{\ast}), H^{1}(\Omega)).$ We are going 
to prove that for a.e. $t\in (0, T_{\ast})$, \begin{equation} z(x,t) = \sigma \kappa(x, t)-\lambda v(x,t) = w(x, t)~ \text{for}~ \mathcal{H}^{N-1}~\text{a.e.}~x\in \Gamma(t).\label{vkappa}\end{equation}
Indeed, from (\ref{singlet}) and 
$\lim_{\e\rightarrow 0} \norm{u^{\e}(t)-\overline{u^{\e}(t)_{\Omega}}}^2_{H^{-1}(\Omega)}=  
\norm{u(t)-\overline{u(t)_{\Omega}}}^2_{H^{-1}(\Omega)},$ one deduces the single-multiplicity
 property of the limiting interface $\Gamma(t)$ on each time 
slice $t$. That is, in the sense of Radon measures 
\begin{equation*}
\left(\frac{\varepsilon\abs{\nabla u^{\varepsilon}}^2}{2} +
\frac{W(u^{\varepsilon})}{\varepsilon}\right)dx \rightharpoonup 2\sigma d\mathcal{H}^{N-1}\lfloor\Gamma(t).
\end{equation*}
Moreover, we have  the uniform bound on the energy $E_{\varepsilon}(u^{\varepsilon}(t))\leq M$ for all $t\in [0, T_{\ast}]$ and 
all $\varepsilon>0$. Combining these facts with the dominated convergence theorem, we get\\
- The single-multiplicity in space-time, i.e, in the sense of Radon measures, 
\begin{equation*}\left(\frac{\varepsilon\abs{\nabla u^{\varepsilon}}^2}{2} +
\frac{W(u^{\varepsilon})}{\varepsilon}\right)dx dt \rightharpoonup 2\sigma d\mathcal{H}^{N-1}\lfloor\Gamma(t) dt. 
\label{constantm4} \end{equation*}
- The limiting equipartition of energy in space-time, i.e, in the sense of Radon measures
\begin{equation*}
 \abs{\frac{\varepsilon \abs{\nabla u^{\varepsilon}}^2}{2}-\frac{W(u^{\varepsilon})}{\varepsilon}} dxdt\rightharpoonup 0.
\label{vanishing4}\end{equation*}
Arguing as in the proof of Lemma 3.1 in \cite{LeCH}, we get (\ref{vkappa}). Now, we pass to the 
limit in the equation $\partial_{t} u^{\varepsilon} = -\Delta w^{\varepsilon}$.  Recalling 
that $w^{\varepsilon}$ weakly converges to $z$ in $L^{2}((0, T_{\ast}), H^{1}(\Omega))$ and that $w^{\varepsilon}$ satisfies 
the zero Neumann boundary condition, we find 
that $2\delta_{\Gamma(s)}\partial_{t}\Gamma(s) = - \Delta z(s)$ in $\Omega\times (0, T_{\ast})$ and 
$\frac{\partial z}{\partial n} =0$ on $\partial\Omega \times (0, T_{\ast})$ in the sense of 
distributions. To see this, fix $t\in (0, T).$ From the assumptions of our Theorem and the dominated convergence theorem, we 
find that $u^{\varepsilon}\rightarrow u$ in $L^{1}(\Omega\times [0, T]) $. It follows that 
$\partial_{t}u^{\varepsilon}(x,s)\rightarrow \partial_{t}u(x,s)$ in the sense of distributions. Denote 
by $\Omega^{+}(s)$ the set \{$x\in\Omega: u(x,s) = 1$\} and recall that 
$\Gamma(s)=\partial\{u(s)=1\}\cap\Omega$ is the interface separating the phases $-1$ and $+1.$
Then, $\partial_{t}u(s)=\partial_{t}(u(s)+1)=\partial_{t} (2\chi_{\Omega^{+}(s)})=2\delta_{\Gamma(s)}\partial_{t}\Gamma(s) = -\Delta z(s).$\\
\h Recall from {\bf 1.} that $2\delta_{\Gamma(s)}\partial_{t}\Gamma(s) = -\Delta w(s)$. Therefore, in the sense of 
distributions, $\Delta (z-w) =0$ in $\Omega\times (0, T_{\ast})$ and $\frac{\partial (z-w)}{\partial n} =0$ on 
$\partial\Omega\times (0, T_{\ast}).$ From (\ref{vkappa}), we conclude that $z = w$ a.e. in $\Omega\times (0, T_{\ast})$ and 
this shows that $w^{\varepsilon}$ converges weakly to $w$ in $L^{2}((0, T_{\ast}), H^{1}(\Omega))$. \\
{\bf 3.} Finally, we now complete the proof of the theorem by showing that $w^{\varepsilon}$ actually 
converges strongly in $L^{2}((0, T_{\ast}), H^{1}(\Omega))$ to $w.$
In fact, because of the equality (\ref{singlet}), the inequality (\ref{strong}) is actually an equality. Therefore 
\begin{equation}\lim_{\varepsilon\rightarrow 0}\int_{0}^{T_{\ast}} \norm{\nabla w^{\varepsilon}(s)}^{2}_{L^{2}(\Omega)} = 
\int_{0}^{T_{\ast}}\int_{\Omega}\abs{\nabla w(x,s)}^2 dxds.
\label{Yconv}
\end{equation}
 Since $\nabla w^{\varepsilon}$ converges weakly 
to $\nabla w$ in $L^{2}((0, T_{\ast}), L^{2}(\Omega))$, we conclude that $\nabla w^{\varepsilon}$ converges strongly 
to $\nabla w$ in $L^{2}((0, T_{\ast}), L^{2}(\Omega)).$ It follows that $w^{\varepsilon}$ converges strongly 
to $w$ in $L^{2}((0, T_{\ast}), H^{1}(\Omega))$ and this completes the proof of Theorem \ref{dynamics}.
 \end{proof}   
\section{Proof of Theorem \ref{sdynamics}}
\label{sthmsec}
\h In this section, we prove Theorem \ref{sdynamics}.
\begin{proof}[Proof of Theorem \ref{sdynamics}]
By (\ref{volholder}), $u^{\e}$ is H\"{o}lder continuous in time. From its radial symmetry and the fact that $\Gamma(t)$
consists of a finite number of spheres, we have the H\"{o}lder continuity 
in time for the limiting
interface $\Gamma(t)$. This together with $(BC)$ implies the existence of $T_{\ast}>0$ such that for all $t\in [0, T_{\ast})$, the spheres 
contained in $\Gamma(t)$ do not collide and 
\begin{myindentpar}{0.5cm}
(BC')The limit measure $\mu(t)$ of $\left(\frac{\varepsilon}{2}\abs{\nabla u^{\e}(t)}^2 +\frac{1}{\varepsilon}W(u^{\e}(t))\right)dx$
(in the sense of Radon measures) does not concentrate on the boundary $\partial\Omega$: $\mu(t)(\partial\Omega) =0$.
\end{myindentpar}
As in (\ref{boundk}), denoting $k^{\e}(t) = \e\Delta u^{\e}(t) - \e^{-1}f(u^{\e}(t))$, we have
\begin{equation*}
 \norm{k^{\varepsilon}(t)}_{H^{1}(\Omega)}\leq 
C(M + \norm{\nabla w^{\varepsilon}(t)}_{L^{2}(\Omega)} ).
\end{equation*}
Integrating from $0$ to $T^{\ast}$, and recalling (\ref{Hminus1}), we obtain
\begin{equation}
 \int_{0}^{T^{\ast}} \norm{k^{\e}(t)}^2_{H^{1}(\Omega)} dt\leq C.
\end{equation}
By Fatou's lemma, for $L^{1}$ a.e $t\in [0, T^{\ast})$, we have
\begin{equation}
\liminf_{\e\rightarrow 0} \norm{k^{\e}(t)}_{H^{1}(\Omega)} <\infty.
\label{Fatou}
\end{equation}
Let $t_{0}\geq 0$ be any sufficiently small number such that (\ref{Fatou}) is satisfied. It suffices to prove the following
\begin{propo}
The limit function $(u^{0},v,w)$ and the interfaces $\Gamma(t)$ satisfy (\ref{NO}) on $[t_{0}, T^{\ast})$. Furthermore, we have 
well-preparedness of the interface $\Gamma(t)$ for all time slice $t\geq t_{0}$, i.e.,
\begin{equation*}
 \lim_{\varepsilon\rightarrow 0}E_{\varepsilon}(u^{\varepsilon}(t))= E(\Gamma(t)).
\end{equation*}
\label{backward}
\end{propo}
Then $(u^{0},v,w)$ and $\Gamma(t)$ satisfy (\ref{NO}) on $[0, T^{\ast})$ with the initial data $\Gamma(0)$ understood as the initial trace: $\lim_{t\searrow 0} \Gamma(t) = \Gamma(0)$. Indeed, for radial solution with interface consisting of a finite number of spheres , the H\"{o}lder continuity in time of $u^{\e}$ in (\ref{volholder}) implies the H\"{o}lder continuity in time of $\Gamma(t)$. Thus the above limit of $\Gamma(t)$ as $t\rightarrow 0$ exists.\\
\h The proof of Proposition \ref{backward} relies on the following theorem, which could be of independent interest.
\begin{theorem}
Let $(u^{\e})$ be a sequence of smooth radially symmetric functions on $\Omega = B_{1}$ such that
\begin{myindentpar}{0.5cm}
 (1) $\frac{\partial u^{\e}}{\partial n} = 0$ on $\partial\Omega$,~~
(2) $\displaystyle\int_{\Omega}\left(\frac{\e}{2}\abs{\nabla u^{\e}}^2 +\frac{1}{\varepsilon}W(u^{\e})\right) dx\leq C $,\\
(3) $\liminf_{\e\rightarrow 0}\norm{\e\Delta u^{\e} - \e^{-1}f(u^{\e})}_{H^{1}(\Omega)} \leq C$.\\
(4) The limit measure $\mu$ of $\left(\frac{\varepsilon}{2}\abs{\nabla u^{\e}}^2 +\frac{1}{\varepsilon}W(u^{\e})\right)dx$
(in the sense of Radon measures) does not concentrate on the boundary $\partial\Omega$: $\mu(\partial\Omega) =0$.
\end{myindentpar}
Then, up to extracting a subsequence, $u^{\e}$ converges in $L^{2}(\Omega)$ to $u\in BV(\Omega, \{-1,1\})$ with interface $\Gamma$ separating
the phases. Then  
\begin{equation}
 \lim_{\e\rightarrow 0} E_{\e}(u^{\e}) = E(\Gamma).
\label{instance}
\end{equation}
\label{instancethm}
\end{theorem}
\begin{remark}
 Our theorem is an elliptic refinement of Chen's result \cite{ChenCH} (Theorem 5.3) for the time-dependent Cahn-Hilliard equation.
\end{remark}
\begin{proof}
 For simplicity, let us denote $k^{\e} = \e\Delta u^{\e} - \e^{-1}f(u^{\e})$ and the discrepancy measure by $\xi^{\e} =
\frac{\e}{2}\abs{\nabla u^{\e}}^2 -\frac{1}{\varepsilon}W(u^{\e} )$. By (1) and (2) and following the argument of the proof
of Theorem 5.1 in \cite{ChenCH}, one can bound the discrepancy measure in term of the Allen-Cahn energy as follows
\begin{equation*}
 \int_{\Omega}\abs{\xi^{\e}} dx \leq C_{1} \left(\delta + \eta + \e + C(\delta,\eta)\sqrt{\e}\right) \int_{\Omega}\left(\frac{\e}{2}\abs{\nabla u^{\e}}^2 +\frac{1}{\varepsilon}W(u^{\e})\right) 
dx,
\end{equation*}
where $\delta,\eta$ are arbitrary small numbers and $C_{1}$ is independent of $\e,\delta,\eta$. Sending first $\e$ to $0$ and then $\delta$ and $\eta$ to $0$, 
we obtain the limiting equipartition of energy
\begin{equation}
  \lim_{\e\rightarrow 0}\int_{\Omega}\abs{\xi^{\e}}  dx=0.
\label{zerodiscrep}
\end{equation}
It is easy to see from $(2)$ that, up to extracting a subsequence, $u^{\e}$ converges in $L^{p}(\Omega)$ ($1\leq p<4$) to $u\in BV(\Omega, \{-1,1\})$ 
with interface $\Gamma$ separating
the phases, see, e. g. \cite{Stern}. Moreover, $\Gamma$ consists of a finite number of spheres with radii $0<r_{1}<r_{2}<\cdots < r_{k}\leq1.$ In the sequel, we will take $p=10/3$. 
The limit measure $\mu$ of $e^{\e} = \left(\frac{\e}{2}\abs{\nabla u^{\e}}^2 +\frac{1}{\varepsilon}W(u^{\e})\right)dx$ contains $\Gamma =
\cup_{i=1}^{k}\partial B_{r_{i}}$ in its support.
Because there is no energy concentrating on the boundary $\partial\Omega$ due to 
$(4)$, we must have $r_{k}<1$. Now we prove that $\mu$ concentrates exactly on $\Gamma$. Indeed, writing
\begin{equation*}
 \frac{\e}{2}\abs{\nabla u^{\e}}^2 +\frac{1}{\varepsilon}W(u^{\e}) = \left(\sqrt{\frac{\e}{2}}\abs{\nabla u^{\e}}
 - \sqrt{\frac{W(u^{\e})}{\e}}\right)^{2} + \abs{\nabla u^{\e}} \sqrt{2 W(u^{\e})}
\end{equation*}
and keeping in mind that $W(u)=\frac{1}{2} (1-u^2)^2$, one has
\begin{equation*}
 \frac{\e}{2}\abs{\nabla u^{\e}}^2 +\frac{1}{\varepsilon}W(u^{\e}) = \left(\sqrt{\frac{\e}{2}}\abs{\nabla u^{\e}}
 - \sqrt{\frac{W(u^{\e})}{\e}}\right)^{2} + \abs{\nabla (u^{\e} - \frac{(u^{\e})^3}{3})}.
\end{equation*}
On the other hand, it is easy to see that
\begin{equation*}
 \left(\sqrt{\frac{\e}{2}}\abs{\nabla u^{\e}}
 - \sqrt{\frac{W(u^{\e})}{\e}}\right)^{2} \leq \abs{\sqrt{\frac{\e}{2}}\abs{\nabla u^{\e}}
 - \sqrt{\frac{W(u^{\e})}{\e}}}\left(\sqrt{\frac{\e}{2}}\abs{\nabla u^{\e}}
 + \sqrt{\frac{W(u^{\e})}{\e}}\right) =\abs{\xi^{\e}}.
\end{equation*}
Therefore, it follows from (\ref{zerodiscrep}) that the limit measure $\mu$ of $e^{\e}$ is that of $\abs{\nabla (u^{\e} - \frac{(u^{\e})^3}{3})} dx.$
Because $u^{\e}$ converges to $u$ in $L^{10/3}(\Omega)$, $u^{\e} - \frac{(u^{\e})^3}{3}$ converges in $L^{10/9}(\Omega)$ to $u - \frac{u^3}{3} = \frac{2}{3} u$ 
where $u\in 
BV(\Omega, \{-1,1\})$. This together with the fact that $u^{\e} - \frac{(u^{\e})^3}{3}$ is radial shows that the limit measure $\mu$ of $\abs{\nabla (u^{\e} - 
\frac{(u^{\e})^3}{3})} dx$ concentrates on the support of $\abs{\nabla u}$. Hence $\mu$ concentrates on $\Gamma = \cup_{i=1}^{k}\partial B_{r_{i}}$.
More precisely, there are numbers $m_{1},\cdots, m_{k}>0$ such that, in the sense of Radon measures
\begin{equation}
 e^{\e} \rightharpoonup \mu= \sum_{i=1}^{k} m_{i}2\sigma \mathcal{H}^{N-1}\lfloor \partial B_{r_{i}}.
\end{equation}
We claim that $m_{j} =1$ for all $j$. Note that the case $m_{j}>1$ for some $j$, if exists, corresponds to the piling up of the interface.\\
The key of the proof is the following identity for $\varphi = (\varphi^{1},\cdots,\varphi^{N})\in (C_{0}^{1}(\Omega))^{N}$
\begin{multline}
 \int_{\Omega}\left(\mathrm{div}\varphi - \sum_{j, k}\frac{\partial_{j}u^{\varepsilon}}
{\abs{\nabla u^{\varepsilon}}}\frac{\partial_{k}u^{\varepsilon}}{\abs{\nabla u^{\varepsilon}}}\partial_{k}\varphi^{j}\right)
\varepsilon\abs{\nabla u^{\varepsilon}}^2 dx = \int_{\Omega}\left(\xi^{\e}\mathrm{div}\varphi - u^{\varepsilon} \mathrm{div}(k^{\e}\varphi)\right) dx.\label{varifold}
\end{multline}
This identity can be obtained by multiplying both sides of the equation 
$k^{\varepsilon} = \varepsilon \Delta u^{\varepsilon}-\varepsilon^{-1}f(u^{\varepsilon})$ by $\nabla u^{\varepsilon}\cdot \varphi$ and 
then integrating by parts twice. \\
\h For any $j$, choose a thin annulus $A_{j}$ around $\partial B_{r_{j}}$ such that $
\left(\cup_{i\neq j}^{k}\partial B_{r_{i}}\right)\cap A_{j} =\emptyset$.\\
Now, fix $j$. Choose $\varphi \in C^{1}_{0}(A_{j})$ to localize (\ref{varifold}). Because the limit measure of $e^{\e}$ has constant
multiplicity $m_{j}$ in $A_{j}$ and by the limiting equipartition of energy (\ref{zerodiscrep}), we observe as in \cite{LeCH} that
\begin{equation*}
 \varepsilon \nabla u^{\varepsilon}\otimes\nabla u^{\varepsilon} dx\lfloor A_{j}\rightharpoonup 2m_{j}\sigma 
\stackrel{\rightarrow}{n}\otimes\stackrel{\rightarrow}{n} \mathcal{H}^{N-1}\lfloor\partial B_{r_{j}}.\label{res}
\end{equation*}
Consequently, letting $\e\rightarrow 0$ in (\ref{varifold}),
we obtain 
\begin{equation}
 2m_{j}\sigma\int_{\partial B_{j}} (\mathrm{div}\varphi - \partial_{k}\varphi^{j}\stackrel{\rightarrow}{n}_{j}\otimes\stackrel{\rightarrow}
{n}_{k})d\mathcal{H}^{N-1} = -\int_{A_{j}} u \mathrm{div}(k\varphi) dx,
\label{AFP}
\end{equation}
where $k$ is the weak limit in $H^{1}(\Omega)$ of $k^{\e}$ and 
$\stackrel{\rightarrow}{n} = (\stackrel{\rightarrow}{n}_{1},\cdots,\stackrel{\rightarrow}{n}_{N})$ is an outward unit 
normal to $\partial B_{r_{j}}$. Applying the divergence theorem to the left hand side of (\ref{AFP}), we get
\begin{equation}
 2m_{j}\sigma \int_{\partial B_{r_{j}}} \varphi \frac{N-1}{r_{j}} \stackrel{\rightarrow}
{n}d\mathcal{H}^{N-1} = -\int_{A_{j}} u \mathrm{div}(k\varphi) dx.
\label{AFP2}
\end{equation}
Now, we are ready to prove the Claim. Fix $j$ where $1\leq j\leq k$. Then $\partial B_{r_{j}} \subset \Gamma$ and $u=1$ on one side of $A_{j}$ and $u=-1$
on the other side of $A_{j}$ (with respect to $\partial B_{r_{j}}$). Using the divergence theorem for the right hand side of (\ref{AFP2}),
one finds that
\begin{equation}
 2m_{j}\sigma \int_{\partial B_{r_{j}}} \varphi \frac{N-1}{r_{j}} \stackrel{\rightarrow}
{n}d\mathcal{H}^{N-1} = 2\int_{\partial B_{r_{j}}} v\varphi \cdot  \stackrel{\rightarrow} {n} d\mathcal{H}^{N-1}.
\end{equation}
Hence $k = m_{j}\frac{\sigma(N-1)}{r_{j}}$ on $\partial B_{r_{j}}$. 
Combining this with Item 3. in Lemma 5.4 of \cite{ChenCH}, which says that on $\partial B_{r_{j}}, 
k = \pm \frac{\sigma(N-1)}{r_{j}}$, gives $m_{j} =1$ and thus completing the proof of the Claim. \\
\h It follows from the Claim that \begin{equation}
 \lim_{\e\rightarrow 0}\int_{\Omega}\left(\frac{\e}{2}\abs{\nabla u^{\e}}^2 +\frac{1}{\varepsilon}W(u^{\e})\right) dx = 2\sigma \mathcal{H}^{N-1}(\Gamma).
\label{AC}
\end{equation}
Furthermore, because $u^{\e}$ converges to $u$ in $L^{2}(\Omega)$, one has
%\begin{equation}
 $\lim_{\e\rightarrow 0} \norm{u^{\e}-\overline{u^{\e}_{\Omega}}}^2_{H^{-1}(\Omega)}=  \norm{u-\overline{u_{\Omega}}}^2_{H^{-1}(\Omega)}.$
%\label{nonlocalterm}
%\end{equation}
Combining this with (\ref{AC}), one obtains (\ref{instance}) as desired.
\end{proof}
Now, we give the proof of Proposition \ref{backward}. For ease of notation and by translating time, we can assume that $t_{0} =0.$ By 
(\ref{Fatou}), (BC') and Theorem \ref{instancethm}, the equation (\ref{OKeq}) has well-prepared initial data. We claim that, for all 
$t\in [0, T^{\ast})$, 
\begin{eqnarray}\liminf_{\varepsilon\rightarrow 0}\int_{\Omega}
\abs{\nabla w^{\varepsilon}(t)}^2 dx \geq  \norm{\sigma\kappa(t) -\lambda v}^{2}_{H_{n}^{1/2}(\Gamma(t))}.
\label{global}\end{eqnarray}
Indeed, we only need to prove inequality for the case the right hand side of (\ref{global}) is finite. Then, as in (\ref{Hminus1}) and 
(\ref{boundk}), we have
\begin{equation}
 \liminf_{\e\rightarrow 0}\left(\norm{w^{\e}(t)}_{H^{1}(\Omega)} + \norm{k^{\e}(t)}_{H^{1}(\Omega)}\right)\leq C.
\end{equation}
Thus, by Theorem \ref{instancethm}, we have
\begin{equation}
\lim_{\e\rightarrow 0}\int_{\Omega}\left(\frac{\e}{2}\abs{\nabla u^{\e}(t)}^2 +\frac{1}{\varepsilon}W(u^{\e}(t))\right) dx = 2\sigma \mathcal{H}^{N-1}(\Gamma(t)).
\label{elliptict}
\end{equation}
Recall that $w^{\e}(t) = k^{\e}(t) -\lambda v^{\e}(t)$. By extracting a subsequence, $w^{\e}(t)$ and $k^{\e}(t)$ converge weakly to $w(t)$ and
$k(t)$ respectively in $H^{1}(\Omega)$. It is well-known \cite{LM} that the single multiplicity of the 
interface $\Gamma(t)$ in (\ref{elliptict}) gives the Gibbs-Thompson relation $k(t) =\sigma \kappa(t)$ on $\Gamma(t)$. Thus 
$w(t) =\sigma \kappa(t) -\lambda v(t)$ on $\Gamma(t)$. Now (\ref{global}) follows as in the proof of the Lemma \ref{mainlemma}. We remark that well-preparedness
of initial data and (\ref{global}) are all we need to complete the proof of Proposition \ref{backward}, following the same lines of argument
as in the proof of Theorem \ref{dynamics}. Thus the proof of Theorem \ref{sdynamics} is also complete.
\end{proof}
\section{Proof of Theorem \ref{transport}}
\label{transsec}
\h In this section, we give the proof of Theorem \ref{transport}.
\begin{proof}[Proof of Theorem \ref{transport}] We recall the following notation for all $s\geq 0$
 \begin{equation*}
  \norm{\cdot}^2_{Y(s)} = 4\norm{\cdot}^2_{H^{-1/2}_{n}(\Gamma(s))}
 \end{equation*}
It follows from the proofs of Theorems \ref{dynamics} and \ref{sdynamics} that for all $t_{0}>0$, we have
\begin{myindentpar}{1cm}
1. Well-preparedness of the evolving interface, i.e, 
\begin{equation}
 \lim_{\e\rightarrow 0} E_{\e}(u^{\e}(t_{0})) = E(u(t_{0}))
\label{tone}
\end{equation}
2. The convergence of the velocity in its natural energy space (cf. (\ref{Yconv}))
 \begin{equation}\lim_{\varepsilon\rightarrow 0}\int_{t_{0}}^{T_{\ast}} \norm{\nabla w^{\varepsilon}(s)}^{2}_{L^{2}(\Omega)} ds = 
\int_{t_{0}}^{T_{\ast}}\int_{\Omega}\abs{\nabla w(x,s)}^2 dxds = \int_{t_{0}}^{T^{\ast}}\norm{\nabla_{Y(s)} E(\Gamma(s))}^2_{Y(s)} ds.
\label{ttwo}
\end{equation}
\end{myindentpar}
For the case of well-prepared initial data, as it can be seen from the proof
of Theorem \ref{dynamics} that (\ref{tone}) and (\ref{ttwo}) also hold for $t_{0}=0$.
The first equality, (\ref{tone}), allows us to construct a deformation presented in Proposition \ref{deform}. The second 
equality, (\ref{ttwo}), allows us to apply the deformation to prove the 
transport estimate stated in (\ref{veloconv}). The proof of
Theorem \ref{transport} will then follow from Lemma \ref{modifiedV} and the transport estimate in Section \ref{transportest}.
\subsection{Construction of the deformation} Our main result in this section is the construction of a deformation in the following
\begin{propo}
Let $(u^{\e})$ be a sequence of smooth functions on $\Omega$ satisfying $\frac{\partial u^{\e}}{\partial n} =0$ on $\partial\Omega$, 
$E_{\e}(u^{\e})\leq M$ and  $u^{\varepsilon} \rightarrow u\in BV(\Omega, \{-1,1\})$ in $L^{2}(\Omega)$ where $u$ has $\Gamma$
as its smooth interface separating the phases $1$ and $-1$. Furthermore, assume that $\Gamma$ consists of a finite 
number of closed, connected hypersurfaces inside $\Omega$ and that
\begin{equation}\lim_{\varepsilon\rightarrow 0}E_{\varepsilon}(u^{\varepsilon})= E(\Gamma).\label{singlet2}
\end{equation} 
Let $V$ be a smooth function defined on $\Gamma$ so that
$V\in H^{-1/2}_{n}(\Gamma).$ Let $w(t)$ be any smooth deformation of $\Gamma$ with normal velocity vector $\bf{V}$ at $t =0$, i.e., $w(t)$ consists of a finite 
number of closed, connected hypersurfaces inside $\Omega$ satisfying
\begin{equation}
 w(0) =\Gamma,~\partial_{t}w(0) = {\bf V}
\end{equation}
where ${\bf V} = V \stackrel{\rightarrow}{n}$. Then, 
we can find $w^{\e}(t)\in C^{1}(\Omega)$ such that $
 w^{\e}(0) = u^{\e}, $ and the following equalities hold
\begin{equation}
 \lim_{\e\rightarrow 0}\norm{\partial_{t} w^{\e}(0)}^2_{H^{-1}_{n}(\Omega)} = \norm{\partial_{t}w(0)}^2_{Y} = \norm{V}^2_{Y}
= 4 \norm{V}^{2}_{H^{-1/2}_{n}(\Gamma)},
\label{veloineq}
\end{equation}
\begin{equation}
 \lim_{\e\rightarrow 0}\left.\frac{d}{dt}\right\rvert_{t=0}E_{\varepsilon}(w^{\e}(t)) = \left.\frac{d}{dt}\right\rvert_{t=0} E(w(t)).
\label{slopeineq}
\end{equation}
\label{deform}
\end{propo}
\begin{proof}
We observe that $V$ being smooth on $\Gamma$ and belonging to
$H^{-1/2}_{n}(\Gamma)$ imply, by Lemma \ref{downhalf}, $
 \int_{\Gamma} V d\mathcal{H}^{N-1}=0.$ In fact, in Lemma \ref{downhalf}, let $u=V$ and $v=1$. Then, by (\ref{seminorm}),
\begin{equation*}
 \int_{\Gamma} V d\mathcal{H}^{N-1}= \langle V,1\rangle_{H_{n}^{-1/2}(\Gamma)\times H_{n}^{1/2}(\Gamma)} = -<V^{\ast},1>_{H_{n}^{1/2}(\Gamma)} = 
- <\nabla \tilde{V^{\ast}}, \nabla \tilde{1}>_{L^{2}(\Omega)}.
\end{equation*}
It follows from (\ref{chartilde}) that $\tilde{1} =1.$ Thus
\begin{equation*}\int_{\Gamma} V d\mathcal{H}^{N-1} = - <\nabla \tilde{V^{\ast}}, \nabla 1>_{L^{2}(\Omega)} =0.
 \end{equation*}
Let us extend the vector field ${\bf V}$ outside $\Gamma$ in such a way that
${\bf V}\in (C^{1}_{c}(\Omega))^{N}$. Let $\Omega^{+} =\{x\in \Omega: u(x) =1\}.$ Then the divergence theorem gives
\begin{equation}
 \int_{\Omega} \mathrm{div} {\bf V} dx =0; ~\int_{\Omega}2\chi_{\Omega^{+}} \mathrm{div} {\bf V} dx =0.
\label{divfree}
\end{equation}
We need the following simple lemma, which also implies the existence of a small perturbation $\partial_{t}{\bf \Gamma}^{\e}$ of 
$\partial_{t}{\bf \Gamma}$ satisfying (\ref{smallde}).
\begin{lemma}
There exists a vector field ${\bf V}^{\e}\in (C^{1}_{0}(\Omega))^{N}$ satisfying the following conditions
(i) $
 \lim_{\e\rightarrow 0}\norm{{\bf V}^{\e}-{\bf V} }_{C^{1}_{0}(\Omega)} =0 $;~~
(ii) $\int_{\Omega} \nabla u^{\e}\cdot {\bf V}^{\e} dx =0$.
\label{modifiedV}
\end{lemma}
\begin{proof}
Let us consider a smooth vector field $\varphi \in (C^{1}_{0}(\Omega))^{N}$ satisfying $
 \int_{\Gamma}\varphi\cdot \stackrel{\rightarrow}{n} d\mathcal{H}^{N-1}\neq 0.$ Let 
$ {\bf V}^{\e} = {\bf V} + h(\e) \varphi $
where $h(\e) \rightarrow 0$ as $\e\rightarrow 0$ to be chosen later. Then ${\bf V}^{\e}\in (C^{1}_{0}(\Omega))^{N}$.
With this choice of ${\bf V}^{\e}$, $(i)$ is clearly satisfied.\\
Concerning $(ii)$, we have, by the divergence theorem and the fact that $\bf{V}^{\e}$ has compact support
\begin{equation*}
 -\int_{\Omega}\nabla u^{\e}\cdot {\bf V}^{\e} dx= -\int_{\Omega} \mathrm{div} (u^{\e} {\bf V}^{\e} ) dx + \int_{\Omega} u^{\e} \mathrm{div} {\bf V}^{\e} dx
= \int_{\Omega} u^{\e} \mathrm{div} {\bf V}^{\e} dx.
\end{equation*}
Because $
 \int_{\Omega}\mathrm{div} {\bf V}^{\e} dx =\int_{\partial\Omega}{\bf V}^{\e} \stackrel{\rightarrow}{n} d\mathcal{H}^{N-1} =0,$
we see that
\begin{equation}
 -\int_{\Omega}\nabla u^{\e}\cdot {\bf V}^{\e} dx = \int_{\Omega} (u^{\e} +1) \mathrm{div} {\bf V}^{\e} dx
= \int_{\Omega} (u^{\e} +1) \mathrm{div} {\bf V} dx + h(\e)\int_{\Omega} (u^{\e} +1) \mathrm{div} \varphi dx.
\label{zerocond}
\end{equation}
Therefore, $(ii)$ will be satisfied by choosing 
\begin{equation*}
 h(\e) = \frac{-\int_{\Omega} (u^{\e} +1) \mathrm{div} {\bf V}dx }{\int_{\Omega} (u^{\e} +1) \mathrm{div} \varphi dx}.
\end{equation*}
It remains to verify that $h(\e) \rightarrow 0$ as $\e\rightarrow 0$. Indeed, because $u^{\e} +1\rightarrow 2\chi_{\Omega^{+}}$ in 
$L^{1}(\Omega)$, the denominator of $h(\e)$, $\int_{\Omega} (u^{\e} +1) \mathrm{div} \varphi dx $ converges to 
$\int_{\Omega}2\chi_{\Omega^{+}} \mathrm{div} \varphi dx  = 2 \int_{\Gamma}\varphi\cdot \stackrel{\rightarrow}{n} d\mathcal{H}^{N-1} \neq 0$, as $\e\rightarrow 0.$ On the other hand,
using (\ref{divfree}), we see that the numerator of $h(\e)$, 
$-\int_{\Omega} (u^{\e} +1) \mathrm{div} {\bf V} dx= -\int_{\Omega} (u^{\e} +1 - 2\chi_{\Omega^{+}}) \mathrm{div} {\bf V} dx \rightarrow 0$
as $\e\rightarrow 0$. As a result, $h(\e) \rightarrow 0$ as $\e\rightarrow 0$.
\end{proof}
Consider $t$ sufficiently small such that the map $
 \chi_{\e, t}(x) = x + t {\bf V}^{\e}(x) $
is a diffeomorphism of $\Omega$ into itself. By the construction of ${\bf V}^{\e}$ in Lemma \ref{modifiedV}, the smallness of $t$
can be chosen independent of $\e$. We define $w^{\e}(x,t)$ as follows
\begin{equation}
 w^{\e}(x,t) = u^{\e}(\chi^{-1}_{\e,t}(x)).
\end{equation}
Let us check that $w^{\e}$ satisfies the desired properties. First, we confirm (\ref{veloineq}) by showing that
\begin{equation*}
 \lim_{\e\rightarrow 0}\norm{\partial_{t} w^{\e}(0)}^2_{H^{-1}_{n}(\Omega)} = \norm{\partial_{t}w(0)}^2_{Y} = \norm{V}^2_{Y}
= \int_{\Gamma}\int_{\Gamma}G(x,y)V(x)V(y) d\mathcal{H}^{N-1}(x)d\mathcal{H}^{N-1}(y)
\end{equation*}
where $G(x,y)$ is the Green's function for $-\Delta$ on $\Omega$
with Neumann boundary conditions. To do so, we start by evaluating
 $\norm{\partial_{t} w^{\e}(0)}^2_{H^{-1}_{n}(\Omega)}. $ 
Note that, for each $x$, we have \begin{equation*}
                                                    x = \chi_{\e, t} (\chi_{\e,t}^{-1}(x)) = \chi_{\e,t}^{-1}(x) + t {\bf V}^{\e} (\chi_{\e,t}^{-1}(x)).
\end{equation*}
Hence 
\begin{equation*}
 0 = \frac{d}{dt}(\chi_{\e, t}^{-1}(x)) + {\bf V}^{\e}(\chi_{\e,t}^{-1}(x)) + t \nabla {\bf V}^{\e}\cdot\frac{d}{dt}(\chi_{\e, t}^{-1}(x)).
\end{equation*}
Evaluating the above equation at $t=0$ and noting that $\chi_{\e,0}^{-1} (x) =x$, one obtains
\begin{equation*}
\left. \frac{d}{dt}\right\rvert_{t=0}(\chi_{\e, t}^{-1}(x)) = -{\bf V}^{\e}(x).
\end{equation*}
Thus 
\begin{equation}
 \partial_{t}w^{\e}(0) = \nabla u^{\e}\cdot \left.\frac{d}{dt}\right\rvert_{t=0}(\chi_{\e, t}^{-1}(x)) = -\nabla u^{\e} \cdot {\bf V}^{\e}.
\label{velodeform}
\end{equation}
By Lemma \ref{modifiedV} $(ii)$, there exists $g_{\e}\in H^{1}(\Omega)$ such that $-\Delta g_{\e} = \nabla u^{\e} \cdot {\bf V}^{\e} = \nabla u_{\ast}^{\e} \cdot {\bf V}^{\e}$ and 
$\frac{\partial g_{\e}}{\partial n} =0$ where we have denoted $u^{\e}_{\ast} = u^{\e} + 1$. 
Then, by the definition of the $H^{-1}_{n}(\Omega)$ norm in Section \ref{grOK} 
\begin{equation*}
 \norm{\partial_{t} w^{\e}(0)}^2_{H^{-1}_{n}(\Omega)} = \norm{\nabla u_{\ast}^{\e} \cdot {\bf V}^{\e}}^2_{H^{-1}_{n}(\Omega)}=
\int_{\Omega} \abs{\nabla g_{\e}}^2 dx. 
\end{equation*}
Now, let $G(x,y)$ be the Green's function for $-\Delta$ on $\Omega$
with Neumann boundary conditions. Then
\begin{equation}
 \int_{\Omega} \abs{\nabla g_{\e}}^2 dx
=\int_{\Omega}\int_{\Omega} G(x,y)\nabla u_{\ast}^{\e}(x) \cdot{\bf V}^{\e}(x) \nabla u_{\ast}^{\e}(y)\cdot {\bf V}^{\e}(y) dxdy.
\label{firstg}
\end{equation}
Using integration by parts
\begin{eqnarray*}
 \int_{\Omega}G(x,y)\nabla u_{\ast}^{\e}(x)\cdot {\bf V}^{\e}(x) dx & =& 
\int_{\Omega}G(x,y)[\mathrm{div}_{x} [u_{\ast}^{\e}(x) {\bf V}^{\e}(x)] - u_{\ast}^{\e}(x)\mathrm{div}_{x}{\bf V}^{\e}(x)] dx\\
&=& -\int_{\Omega} \left( \nabla_{x} G(x,y)\cdot {\bf V}^{\e}(x) u_{\ast}^{\e}(x) + u_{\ast}^{\e}(x)G(x,y)\mathrm{div}_{x}{\bf V}^{\e}(x)\right) dx\\ &\equiv& -H(y).
\end{eqnarray*}
Using integration by parts one more time
\begin{equation*}
 \int_{\Omega}H(y)\nabla u_{\ast}^{\e}(y)\cdot {\bf V}^{\e}(y)  dy = -\int_{\Omega} \left( \nabla_{y} H(y)\cdot {\bf V}^{\e}(y) u_{\ast}^{\e}(y) + 
u_{\ast}^{\e}(y)H(y)\mathrm{div}_{y}{\bf V}^{\e}(y)\right) dy.
\end{equation*}
Thus, (\ref{firstg}) gives
\begin{multline}
 \int_{\Omega} \abs{\nabla g_{\e}}^2 dx=  -\int_{\Omega}H(y)\nabla u_{\ast}^{\e}(y)\cdot {\bf V}^{\e}(y) dy\\ = 
\int_{\Omega} \left( \nabla_{y} H(y) \cdot {\bf V}^{\e}(y) u_{\ast}^{\e}(y)+ 
u_{\ast}^{\e}(y)H(y)\mathrm{div}_{y}{\bf V}^{\e}(y)\right) dy\\=
\int_{\Omega} u_{\ast}^{\e}(y){\bf V}^{\e}(y) dy\cdot \int_{\Omega} \left( \nabla_{y}\left[\nabla_{x} G(x,y)\cdot {\bf V}^{\e}(x)u_{\ast}^{\e}(x)\right] + 
u_{\ast}^{\e}(x)\nabla_{y}G(x,y)\mathrm{div}_{x}{\bf V}^{\e}(x)\right) dx \\
+ 
\int_{\Omega}u_{\ast}^{\e}(y)\mathrm{div}_{y}{\bf V}^{\e}(y) dy\int_{\Omega} \left( \nabla_{x} G(x,y) {\bf V}^{\e}(x)u_{\ast}^{\e}(x) + 
u_{\ast}^{\e}(x)G(x,y)\mathrm{div}_{x}{\bf V}^{\e}(x)\right) dx.
\label{beforelim}
\end{multline}
Letting $\e\rightarrow 0$ in (\ref{beforelim}), taking into account Lemma \ref{modifiedV} (i) and the fact that  $
 u^{\e}_{\ast} = u^{\e} +1 \rightarrow 2\chi_{\Omega^{+}}$ in $L^{1}(\Omega)$ as $\e\rightarrow 0$,  we find that
\begin{multline*}
 \lim_{\e\rightarrow 0}\norm{\partial_{t} w^{\e}(0)}^2_{H^{-1}_{n}(\Omega)} \\=4
\int_{\Omega} \chi_{\Omega^{+}}(y){\bf V}(y) dy \cdot\int_{\Omega} \left( \nabla_{y}\left[\nabla_{x} G(x,y) 
\chi_{\Omega^{+}}(x){\bf V}(x)\right] + 
\chi_{\Omega^{+}}(x)\nabla_{y}G(x,y)\mathrm{div}_{x}{\bf V}(x)\right) dx \\
+ 4\int_{\Omega}\chi_{\Omega^{+}}(y)\mathrm{div}_{y}{\bf V}(y) dy\int_{\Omega} \left( \nabla_{x} G(x,y)\cdot {\bf V}(x)\chi_{\Omega^{+}}(x) + 
\chi_{\Omega^{+}}(x) G(x,y)\mathrm{div}_{x}{\bf V}(x)\right) dx\\
= 4\int_{\Omega^{+}}\left ({\bf V}(y) \cdot \nabla_{y} M(y) + \mathrm{div}_{y} {\bf V}(y) M(y) \right)dy,
\end{multline*}
where
\begin{eqnarray*}
 M (y)&=& \int_{\Omega^{+}}\left( \nabla_{x}G(x,y)\cdot {\bf V}(x) + \mathrm{div}_{x}{\bf V}(x) G(x,y)\right) dx\\ &=& \int_{\Omega^{+}} \mathrm{div}_{x}
\left (G(x,y){\bf V}(x)\right)
 dx 
= \int_{\Gamma}G(x,y){\bf V}(x)\cdot \stackrel{\rightarrow}{n} d\mathcal{H}^{N-1}(x)\\
&=& \int_{\Gamma}G(x,y)V(x) d\mathcal{H}^{N-1}(x).
\end{eqnarray*}
It follows that
\begin{eqnarray*}
 \int_{\Omega^{+}}\left({\bf V}(y) \cdot\nabla_{y} M(y) + \mathrm{div}_{y} 
{\bf V}(y) M(y)\right) dy  &=& \int_{\Omega^{+}} \mathrm{div}_{y} ({\bf V}(y)M(y)) dy\\&=& 
\int_{\Gamma}M(y){\bf V}(y) \cdot \stackrel{\rightarrow}{n}d\mathcal{H}^{N-1}(y) \\&=& \int_{\Gamma}V(y)M(y)d\mathcal{H}^{N-1}(y)\\
& =& \int_{\Gamma}\int_{\Gamma}G(x,y)V(x)V(y) d\mathcal{H}^{N-1}(x) d\mathcal{H}^{N-1}(y).
\end{eqnarray*}
Hence
\begin{equation}
 \lim_{\e\rightarrow 0}\norm{\partial_{t} w^{\e}(0)}^2_{H^{-1}_{n}(\Omega)} = 4\int_{\Gamma}\int_{\Gamma}G(x,y)V(x)V(y) d\mathcal{H}^{N-1}(x)d\mathcal{H}^{N-1}(y).
\label{veloev}
\end{equation}
Now, we will express $\norm{V}^{2}_{H^{-1/2}_{n}(\Gamma)}$ in terms of
the Green function $G(x,y)$ and $V$. To do this, let us denote $V^{\ast} = \Delta^{-1}_{\Gamma}V$ as in Lemma \ref{downhalf}. Then
$\Delta_{\Gamma}V^{\ast} = V$ and 
\begin{equation}
 \norm{V}^{2}_{H^{-1/2}_{n}(\Gamma)} = \norm{V^{\ast}}^{2}_{H^{1/2}_{n}(\Gamma)}.
\label{velolim1}
\end{equation}
Recall from (\ref{seminorm}) that
\begin{equation}
 \norm{V^{\ast}}^{2}_{H^{1/2}_{n}(\Gamma)} = \norm{\nabla \tilde{V^{\ast}}}^2_{L^{2}(\Omega)}
\label{velolim2}
\end{equation}
where $\tilde{V^{\ast}}\in H^{1}(\Omega)$ satisfying $\frac{\partial \tilde{V^{\ast}}}{\partial n} =0$ on $\partial\Omega$
and by (\ref{twoLap}), $\Delta \tilde{V^{\ast}} =
\Delta_{\Gamma}(V^{\ast})\delta_{\Gamma} = V\delta_{\Gamma}$. 
Thus, there is a constant $C$ such that 
\begin{equation*}
 \tilde{V^{\ast}} (x) = -\int_{\Omega} G(x,y)V(y)\delta_{\Gamma}(y) dy + C
\end{equation*}
and therefore,
\begin{eqnarray*}
 \norm{\nabla \tilde{V^{\ast}}}^2_{L^{2}(\Omega)}& =& 
\int_{\Omega}\int_{\Omega} G(x,y)V(x)\delta_{\Gamma}(x)V(y)\delta_{\Gamma}(y) dxdy \\
&=& \int_{\Gamma}\int_{\Gamma}G(x,y)V(x)V(y) d\mathcal{H}^{N-1}(x)d\mathcal{H}^{N-1}(y).
\end{eqnarray*}
Combining the above equality with (\ref{velolim1}) and (\ref{velolim2}), we get
\begin{equation}
 \norm{V}^{2}_{H^{-1/2}_{n}(\Gamma)} = \int_{\Gamma}\int_{\Gamma}G(x,y)V(x)V(y) d\mathcal{H}^{N-1}(x)d\mathcal{H}^{N-1}(y).
\label{velolim4}
\end{equation}
From (\ref{veloev}) and (\ref{velolim4}), we obtain (\ref{veloineq}). \\
\h Next, we prove (\ref{slopeineq}) by establishing 
\begin{equation}
\lim_{\e\rightarrow 0}\left.\frac{d}{dt}\right\rvert_{t=0}\int_{\Omega}\left(\frac{\varepsilon}{2}\abs{\nabla w^{\e}(t) }^2 +
\frac{1}{2\varepsilon}(1-\abs{w^{\e}(t)}^2)^2\right)dx = \left.\frac{d}{dt}\right\rvert_{t=0} \sigma\int_{\Omega}\abs{\nabla w(t)}
\label{localterm}
\end{equation}
and 
\begin{equation}
\lim_{\e\rightarrow 0}\left.\frac{d}{dt}\right\rvert_{t=0}\frac{1}{2}
 \norm{w^{\e}(t)-\overline{w(t)^{\e}_{\Omega}}}^2_{H^{-1}(\Omega)}
 = \left.\frac{d}{dt}\right\rvert_{t=0} \frac{1}{2} \norm{w(t)-\overline{w(t)_{\Omega}}}^2_{H^{-1}(\Omega)}.
\label{nonlocals}
\end{equation}
We first prove (\ref{localterm}). Let us denote
\begin{equation*}
E^{loc}_{\e} (w) =\int_{\Omega}\left(\frac{\varepsilon}{2}\abs{\nabla w}^{2} +
\frac{1}{2\varepsilon}(1-\abs{w}^2)^2\right) dx.
\end{equation*}
We start by evaluating $
 \left.\frac{d}{dt}\right\rvert_{t=0}E^{loc}_{\varepsilon}(w^{\e}(t)). $ In view of the definition of 
$w^{\e}(x,t) = u^{\e}(\chi^{-1}_{\e,t}(x))$, 
with the change of variables $y =\chi_{t}(x)$, we have 
\begin{equation*}
 E_{\varepsilon}(w^{\e}(t))=\int_{\Omega}\left(\frac{\varepsilon}{2}\abs{\nabla u^{\e}\cdot \nabla \chi_{\e,t}^{-1}(\chi_{\e,t}(x))}^2 +
\frac{1}{2\varepsilon}(1-\abs{u^{\e}}^2)^2\right)\abs{\mathrm{det} \nabla \chi_{\e,t}(x)} dx.
\end{equation*}
Since
\begin{equation*}
\nabla \chi_{\e,t}^{-1}(\chi_{\e,t}(x)) = [I + t\nabla {\bf V}^{\e}(x)]^{-1} = I -t\nabla {\bf V}^{\e}(x) + o(t), 
\end{equation*}
\begin{equation*}
 \mathrm{det} \nabla \chi_{\e,t}(x) = \mathrm{det} (I + t\nabla {\bf V}^{\e}(x)) = 1 + t\mathrm{div} {\bf V}^{\e} + o(t)
\end{equation*}
we obtain after a simple calculation
\begin{multline*}
 E^{loc}_{\varepsilon}(w^{\e}(t)) 
= \int_{\Omega}\left(\frac{\varepsilon}{2}\abs{\nabla u^{\e} }^2 +
\frac{1}{2\varepsilon}(1-\abs{u^{\e}}^2)^2\right) (1 + t\mathrm{div}{ \bf V}^{\e}) dx \\-
\int_{\Omega} \e t <\nabla u^{\e}, \nabla u^{\e}\cdot \nabla {\bf V}^{\e}(x)> dx
 + o(t).
\end{multline*}
Therefore
\begin{equation*}
 \left.\frac{d}{dt}\right\rvert_{t=0}E^{loc}_{\varepsilon}(w^{\e}(t)) = \int_{\Omega}\left(\frac{\varepsilon}{2}\abs{\nabla u^{\e} }^2 +
\frac{1}{2\varepsilon}(1-\abs{u^{\e}}^2)^2\right) \mathrm{div}{ \bf V}^{\e} dx -
\int_{\Omega} \e  <\nabla u^{\e}, \nabla u^{\e}\cdot \nabla {\bf V}^{\e}(x)> dx.
\end{equation*}
\h We note that the convergence (\ref{singlet2}) corresponds to the case of single multiplicity of the limiting interface $\Gamma$. Now, the work of
Reshetnyak \cite{Res} (see also \cite{LM}) tells us that 
\begin{equation*}
\varepsilon \nabla u^{\varepsilon}\otimes\nabla u^{\varepsilon} dx\rightharpoonup 2\sigma 
\stackrel{\rightarrow}{n}\otimes\stackrel{\rightarrow}{n} \mathcal{H}^{N-1}\lfloor\Gamma. 
\end{equation*}
Thus, denoting ${\bf H}$ = $\kappa \stackrel{\rightarrow}{n}$ the mean curvature vector of $\Gamma$, we can now calculate, using
Lemma \ref{modifiedV} $(i)$, that
\begin{equation}
 \lim_{\e\rightarrow 0}  \left.\frac{d}{dt}\right\rvert_{t=0}E^{loc}_{\varepsilon}(w^{\e}(t)) =
\int_{\Gamma} 2\sigma \left( \mathrm{div} {\bf V}
-<\stackrel{\rightarrow}{n}, \stackrel{\rightarrow}{n}\cdot \nabla {\bf V}>\right) d\mathcal{H}^{N-1} = 
-2\sigma<{\bf H},{\bf V}>_{L^{2}(\Gamma)}.
\label{lim1}
\end{equation}
On the other hand, we have
\begin{equation}
\left.\frac{d}{dt}\right\rvert_{t=0}\sigma\int_{\Omega}\abs{\nabla w(t)} = -2\sigma<{\bf H},{\bf V}>_{L^{2}(\Gamma)}.
\label{deformcurv}\end{equation}
\h Therefore, (\ref{localterm}) follows from (\ref{lim1}) and (\ref{deformcurv}).\\
Thus, to obtain (\ref{slopeineq}), it remains to establish (\ref{nonlocals}). Let $v(t) =\Delta^{-1}(w(t)-\overline{w(t)_{\Omega}})$. Then, we
recall from (\ref{secondterm}) that
\begin{equation}
\left.\frac{d}{dt}\right\rvert_{t=0} \frac{1}{2}
\norm{w(t)-\overline{w(t)_{\Omega}}}^2_{H^{-1}(\Omega)} = \left.\frac{d}{dt}\right\rvert_{t=0} \frac{1}{2} 
\norm{\nabla v(t)}^{2}_{L^{2}(\Omega)}= 2<v, V>_{L^{2}(\Gamma)}
\label{wterm}
\end{equation}
where $v = \Delta^{-1}(u-\overline{u_{\Omega}})$. As in the proof of (\ref{secondterm}), we see that
\begin{multline*}
\norm{w^{\e}(t)-\overline{w^{\e}(t)_{\Omega}}}^2_{H^{-1}(\Omega)}  =
\int_{\Omega}\int_{\Omega} G(x,y)(w^{\e}(x,t)-\overline{w^{\e}_{\Omega}(t)})(w^{\e}(y,t)-\overline{w^{\e}_{\Omega}(t)})dxdy.
\label{gradweq}
\end{multline*}
Differentiating, we get
\begin{multline}
\left.\frac{d}{dt}\right\rvert_{t=0}\frac{1}{2}\norm{w^{\e}(t)-\overline{w^{\e}(t)_{\Omega}}}^2_{H^{-1}(\Omega)} \\= \int_{\Omega}\int_{\Omega} G(x,y)(w^{\e}(x,0)-\overline{w^{\e}_{\Omega}(0)})\left.\frac{d}{dt}\right\rvert_{t=0}\left(w^{\e}(y,t)-\overline{w^{\e}_{\Omega}(t)})\right)dxdy
\end{multline}
By (\ref{velodeform}) and Lemma \ref{modifiedV}$(ii)$,
\begin{equation}
\left.\frac{d}{dt}\right\rvert_{t=0} \overline{w^{\e}(t)_{\Omega}} = -\overline{\left(\nabla u^{\e}\cdot {\bf V}^{\e}\right)_{\Omega}} =0.
\end{equation}
Let us denote $v^{\e} = \Delta^{-1}(u^{\e}-\overline{u^{\e}_{\Omega}})$. Because $w^{\e}(x,0) = u^{\e}(x)$, there is some constant $c^{\e}$ such that
\begin{equation}
\int_{\Omega} G(x,y)(w^{\e}(x,0)-\overline{w^{\e}_{\Omega}(0)})dy = v^{\e}(x) + c^{\e}.
\end{equation}
Now, one has, using Lemma \ref{modifiedV} (ii) again,
\begin{multline*}
\left.\frac{d}{dt}\right\rvert_{t=0}\frac{1}{2}\norm{w^{\e}(t)-\overline{w^{\e}(t)_{\Omega}}}^2_{H^{-1}(\Omega)}  = \int_{\Omega}(v^{\e} + c^{\e}) (-\nabla u^{\e}\cdot {\bf V}^{\e}) dx = 
 \int_{\Omega}v^{\e}(-\nabla u^{\e}\cdot {\bf V}^{\e}) dx.
\end{multline*}
Integrating by parts gives
\begin{equation}
 \int_{\Omega}v^{\e}(-\nabla u^{\e}\cdot {\bf V}^{\e}) dx =  \int_{\Omega}v^{\e}(-\nabla (u^{\e} +1)\cdot {\bf V}^{\e}) dx = 
\int_{\Omega} (u^{\e} +1) \mathrm{div} (v^{\e}{\bf V}^{\e}) dx.
\end{equation}
Letting $\e\rightarrow 0$, one obtains
\begin{multline}
\lim_{\e\rightarrow 0}\left.\frac{d}{dt}\right\rvert_{t=0}\frac{1}{2}\norm{w^{\e}(t)-\overline{w(t)^{\e}_{\Omega}}}^2_{H^{-1}(\Omega)}\\ = 
\lim_{\e\rightarrow 0} \int_{\Omega} (u^{\e} +1) \mathrm{div} (v^{\e}{\bf V}^{\e}) dx= \int_{\Omega}2\chi_{\Omega^{+}} \mathrm{div} (v{\bf V}) dx
= 2\int_{\Gamma} v{\bf V}\cdot \stackrel{\rightarrow}{n} d\mathcal{H}^{N-1}= 2<v, V>_{L^{2}(\Gamma)}
\end{multline}
and (\ref{nonlocals}) follows.
\end{proof}
\subsection{Transport estimate}
\label{transportest} In this section, we prove the existence of a small perturbation $\partial_{t}{\bf \Gamma}^{\e}$ of 
$\partial_{t}{\bf \Gamma}$ satisfying (\ref{smallde})- (\ref{veloconv}) and thus completing the proof of 
Theorem \ref{transport}. Fix $t_{1}>0$. Let $t_{0}\in [t_{1}, T^{\ast}).$ Then, the smoothness of $\Gamma(t_{0})$ implies that of 
\begin{equation*}
 \nabla_{Y(t_{0})} E(\Gamma(t_{0})) = \frac{1}{2}\Delta_{\Gamma(t_{0})} (\sigma \kappa(t_{0})- \lambda v(t_{0}))\stackrel{\rightarrow}{n}
\end{equation*}
where $\stackrel{\rightarrow}{n}$ is the unit outernormal vector to $\Gamma(t_{0})$. 
By (\ref{tone}) and Proposition \ref{deform}, for any $z$ defined in a neighborhood of $t_{0}$ satisfying
$
 z(t_0) =\Gamma(t_{0}), \partial_{t} z(t_{0}) =-\nabla_{Y(t_{0})} E(\Gamma(t_{0})),
$
there exists $z^{\e}(t) = z^{\e}_{t_{0}}(t)$ such that $z^{\e}(t_{0}) = u^{\e}(t_{0})$,
\begin{equation}
 \limsup_{\e\rightarrow 0}\norm{\partial_{t} z^{\e}(t_{0})}^2_{X_{\e}} = \norm{\partial_{t}z(0)}^2_{Y(t_{0})} = 
\norm{\nabla_{Y(t_{0})} E(\Gamma(t_{0}))}^2_{Y(t_{0})}
\label{lim2}
\end{equation}
and
\begin{equation}
 \lim_{\e\rightarrow 0}\left.\frac{d}{dt}\right\rvert_{t=t_{0}}E_{\varepsilon}(z^{\e}(t)) = \left.\frac{d}{dt}\right\rvert_{t=t_{0}} E(z(t)).
\label{lim3}
\end{equation}
Here we recall that $X_{\e} = H^{-1}_{n}(\Omega).$\\
\h In the following, we will use the notation $\partial_{t}z^{\e}(t_{0}) \equiv \partial_{t}z^{\e}_{t_{0}}(t_{0})$. Note
 that (\ref{lim3}) implies
\begin{eqnarray}
 \lim_{\e\rightarrow 0}<\nabla_{X_{\e}} E_{\e}(u^{\e}(t_{0})), \partial_{t}z^{\e}(t_{0})>_{X_{\e}} &=& 
<\nabla_{Y(t_{0})}E(z(t_{0})), \partial_{t}z(t_{0})>_{Y(t_{0})} \nonumber \\&=& - \norm{\nabla_{Y(t_{0})}E(\Gamma(t_{0}))}^2_{Y(t_{0})}.
\label{lim4}
\end{eqnarray}
Now, upon expanding
\begin{multline*}
 \int_{t_{1}}^{T^{\ast}} \norm{\nabla_{X_{\e}}E_{\e}(u^{\e}) + \partial_{t}z^{\e}(t)}^2_{X_{\e}} dt = 
\int_{t_{1}}^{T^{\ast}} \norm{\nabla_{X_{\e}}E_{\e}(u^{\e}) }^2_{X_{\e}} dt + 
\int_{t}^{T^{\ast}} \norm{ \partial_{t}z^{\e}(t)}^2_{X_{\e}} dt\\+
\int_{t_{1}}^{T^{\ast}} 2<\nabla_{X_{\e}}E_{\e}(u^{\e}) , \partial_{t}z^{\e}(t)>_{X_{\e}} dt,
\end{multline*}
and letting $\e\rightarrow 0$, and using (\ref{ttwo}), (\ref{lim2}) and (\ref{lim4}), we find that
\begin{multline*}
 \lim_{\e\rightarrow 0} \int_{t_{1}}^{T^{\ast}} \norm{\nabla_{X_{\e}}E_{\e}(u^{\e}) + \partial_{t}z^{\e}(t)}^2_{X_{\e}} dt
\\= \int_{t_{1}}^{T^{\ast}} \left(\norm{\nabla_{Y(t)} E(\Gamma(t))}^2_{Y(t)}  + \norm{\nabla_{Y(t)} E(\Gamma(t))}^2_{Y(t)} -
 2\norm{\nabla_{Y(t)} E(\Gamma(t))}^2_{Y(t)}\right) dt =0.
\end{multline*}
This combined with the equation $\partial_{t}u^{\e} = - \nabla_{X_{\e}}E_{\e}(u^{\e}) $ shows that
\begin{equation}
 \lim_{\e\rightarrow 0} \int_{t_{1}}^{T^{\ast}} \norm{\partial_{t}u^{\e}-\partial_{t}z^{\e}(t)}^2_{X_{\e}} dt =0.
\label{concludeest}
\end{equation}
Recall from the construction of $z^{\e}(x,t)$, as in Proposition \ref{deform}, that $
\partial_{t}z^{\e}(x,t) =-\nabla u^{\e}\cdot {\bf V}^{\e} $ (see (\ref{velodeform})). Here ${\bf V}^{\e}$ is a small perturbation of the vector field ${\bf V}$ 
satisfying
${\bf V} =\partial_{t} w(t) = -\nabla_{Y}E(\Gamma(t))= 
(\partial_{t}\Gamma)\stackrel{\rightarrow}{n}$ on $\Gamma(t)$ in the sense that $
 \lim_{\e\rightarrow 0} \norm{{\bf V}^{\e}-{\bf V}}_{C^{1}_{0}(\Omega)} = 0.$
Thus, in terms of the notations of Theorem \ref{transport}, 
$\partial_{t}z^{\e}(x,t) = -\nabla u^{\e}\cdot \partial_{t}{\bf \Gamma}^{\e}$ and (\ref{smallde}) is satisfied. Consequently, we get from 
the estimate (\ref{concludeest}) that 
\begin{equation*}
 \lim_{\e\rightarrow 0} \int_{t_{1}}^{T^{\ast}} \norm{\partial_{t}u^{\e} + \nabla u^{\e}\cdot\partial_{t}{\bf \Gamma}^{\e}}^2_{X_{\e}} dt =0.
\end{equation*}
Therefore, we have proved (\ref{veloconv}) and Theorem \ref{transport}.

\end{proof}

{} 

\end{document}